\documentclass[11pt]{amsart}
\usepackage{amsmath, amsthm, amssymb, amsfonts, verbatim, color}
\usepackage[pdftex]{graphicx}
\usepackage[bookmarks]{hyperref}
\reversemarginpar

\definecolor{darkgreen}{rgb}{0,0.55,0}

\newtheorem{thm}{Theorem}[section]

\newtheorem{lem}[thm]{Lemma}
\newtheorem{prop}[thm]{Proposition}
\theoremstyle{definition}

\theoremstyle{remark}
\newtheorem{rem}[thm]{Remark}

\numberwithin{equation}{section}
\numberwithin{thm}{section}

 \newcommand{\norm}[1]{\left\Vert#1\right\Vert}
\newcommand{\ip}[1]{\langle{#1}\rangle}

\newcommand{\abs}[1]{\left\vert#1\right\vert}
\newcommand{\R}{{\mathbb{R}}}
\newcommand{\C}{{\mathbb{C}}}
\newcommand{\Z}{{\mathbb{Z}}}

\def\calE{\mathcal E}
\def\calN{\mathcal N}
\def\calU{\mathcal U}

\newcommand{\yt}{{y^\tau}}
\newcommand{\yn}{{y^\nu}}

\def\be{\begin{equation}}
\def\ee{\end{equation}}

\def\g{\sqrt{-g}}
\def\a{\alpha}
\def\b{\beta}
\def\e{\epsilon}
\def\vp{\varphi}
\def\ga{\gamma}
\def\L{\mathcal L}
\def\A{\mathcal A}
\def\F{\mathcal F}
\def\D{\mathcal D}

\begin{document}
\title[Defects in the abelian Higgs model]{Topological defects in the abelian Higgs model}
\author[Czubak and Jerrard]{M. Czubak and R.L. Jerrard}
\begin{abstract}
 We give a rigorous description of the dynamics of the Nielsen-Olesen vortex line.  In particular,  given a worldsheet of a string, we construct initial data such that the corresponding solution of the abelian Higgs model will concentrate near the evolution of the string.  Moreover, the constructed solution stays close to the Nielsen-Olesen vortex solution.  
\end{abstract}
\date{\today}
\maketitle

\section{Introduction}
In 1973 Nielsen and Olesen \cite{NielsenOlesen} conjectured a relationship between the abelian Higgs model  in $\R^{1+3}$ and the Nambu-Goto action. 
In this paper we show that their conjecture follows from a conjecture
of Jaffe and Taubes about the 2-dimensional Euclidean abelian Higgs model.
In particular, since the Jaffe-Taubes conjecture is known to hold for a range of values of
a coupling parameter appearing in the abelian Higgs model, our results show that
the Nielsen-Olesen scenario holds in these situations.
% In this article we provide a precise formulation of this conjecture and we prove it in certain cases.

The abelian Higgs model (see \eqref{lag} below) arises in various branches of physics: in high-energy physics, as perhaps the simplest Yang-Mills-Higgs theory; in solid-state physics, in connection with superconductivity; and in cosmology where, for reasons stemming from its relevance to high-energy physics, it provides a basis for studies of the possible behavior of cosmic strings, should any such objects exist.

The Nambu-Goto action  of a $(1+1)$-dimensional string in $(1+3)$-dimensional Minkowski space is (proportional to) the Minkowski area of its worldsheet, see \eqref{NambuGoto} below for a precise formulation.  The associated equations of motion are exactly
the condition that the Minkowski mean curvature of the worldsheet vanishes.
This is  the simplest natural model for the relativistic dynamics of a string in Minkowski space.
We refer to a solution of the equations of motion as  a ``timelike Minkowski minimal surface",
by analogy with ordinary (Euclidean) minimal surfaces.
The action is due to Nambu \cite{Nambu} and Goto \cite{Goto}, and it has its origins in the early days of string theory, as a description
of the evolution of a closed (dual) string. 
See \cite{Goddard08} for a nice historical perspective.  

The relationship between these two models proposed in \cite{NielsenOlesen} is
that solutions of the  abelian Higgs model exhibit, for suitable initial data, 
features known as vortex lines that, Nielsen and Olesen argued, should sweep
out worldsheets that are approximately governed by the Nambu-Goto action.
This proposal has subsequently been investigated particularly intensively by cosmologists
interested in possible cosmic strings, starting with work of Kibble \cite{Kibble76}. 
Models for cosmic strings assume that some form of Yang-Mills-Higgs (YMH) equation, perhaps arising from some yet-unknown grand unified theory, is relevant to descriptions of the distribution of matter in the universe.
One can associate to a YMH model an object called the ``vacuum manifold", and
it is believed that qualitative features of solutions known as topological defects are determined by the topology of the vacuum manifold. 
In particular, string-like defects are expected to form
when the vacuum manifold has a nontrivial fundamental group. The
abelian Higgs model, for which the vacuum manifold is given by $S^1$,
provides the simplest case
of this scenario, and it is thus studied as a useful prototype for more general %, or at least more complicated, 
models whose vacuum manifold is not simply connected.

%Another possibility that has been discussed in the cosmological literature is that the $\Pi_0(\mathcal M)$ is nontrivial, or in other words, that the vacuum manifold is disconnected.  Indeed, Kibble argued in \cite{Kibble76} that if this were the case, the cosmological effects would be highly conspicuous, and the fact that these effects have not been observed is strong evidence that the actual YMH model relevant to descriptions of the universe must have a connected vacuum manifold. A caricature of a YMH model with a disconnected vacuum manifold   is the {\em real scalar} wave equation $\Box u + \e^{-2}(u^2-1)u= 0$, known as the Goldstone model  \cite{Goldstone}, solutions of which are expected, since the work of Kibble or before,  to exhibit features known as domain walls that are governed by the Nambu-Goto action. Note however that the Goldstone model is not a gauge theory.
%(This was recently proved in \cite{Jerrard09}, see below).

There is a large body of mathematics describing strings and other 
defects in solutions of elliptic and parabolic equations with vacuum
manifolds that are either disconnected or non-simply-connected. 
References and a
more detailed discussion may be found in \cite{Jerrard09}.

On the other hand, there is not a great deal of rigorous mathematical work describing dynamics of topological defects in
nonlinear hyperbolic equations, and most of it deals with defects
that can be thought of as point particles or $0$-dimensional defects, 
see for example \cite{StuartD, JerrardGL99, Lin99, GSeffective}.
%(cite more works by Stuart?).
Higher-dimensional defects are however treated in
\cite{BGN} and \cite{Jerrard09}. In particular, the latter work proves that
topological defects in certain semilinear hyperbolic equations,
including a non-gauged analog
of the abelian Higgs model, do indeed approximately sweep
out timelike minimal surfaces for suitable initial data. This covers the case of the \emph{domain wall}, one of the basic examples of topological defects considered by cosmologists, associated to the real scalar equation $\Box u + \e^{-2}(u^2-1)u= 0$.  Cosmic strings, as in the abelian Higgs model, have a much richer mathematical structure, and are also considered more likely to be present in our universe than domain walls.

% The other ones are strings in the abelian Higgs model, monopoles in the nonabelian YMH and textures in the Skyrme model (see for example \cite{vs}).  Therefore, the next natural step in the mathematical study of the topological defects in cosmology is to consider strings in the abelian Higgs model.

The basic scheme we use here draws on that developed
in \cite{Jerrard09}.  To show that
this scheme works for a gauge theory such as the abelian Higgs model
we must, among other things, formulate and establish suitable stability
estimates, relating  energy and vorticity for the 2-dimensional 
{\em Euclidean} abelian Higgs model, and a large part of our
work is devoted to these tasks. 
% To show that
%this scheme still works
%for a gauge theory such as the abelian Higgs model
%we must, among other things, 

We next present some necessary background about the abelian Higgs model,  the Nambu-Goto action, the 2d Euclidean abelian Higgs model and the Jaffe-Taubes conjecture, and
normal coordinates around a string. With this
done, we will finally state our main result.
%The introduction concludes with some remarks about well-posedness for the abelian Higgs model and .

\subsection{The abelian Higgs model}

We will write the Lagrangian for the abelian Higgs model in the form
\be
\mathcal L(\vp,\A)=\frac 12  \ip{D_{\a}\vp,D^{\alpha}\vp}+\frac {\e^2}4 \F_{\a\b}\F^{\a\b}+\frac{\lambda}{8\e^2}(\abs{\vp}^{2}-1)^{2}.
\label{lag}\ee
Here
\[
\vp: \R^{1+3}\to \C,\quad\quad
\A\mbox{  a $1$-form with components }\A_\a: \R^{3+1}\rightarrow \R
, \ \ \alpha\in 0,\ldots, 3,
\]
$D_\alpha$ denotes the covariant derivative
$
D_{\alpha}\vp := (\partial_{\alpha}-i\A_{\alpha})\vp,
$
and $\F := d\A$, so that
\[
\F_{\a\b} = \partial_a \A_\b - \partial_\b \A_\a, \qquad\a,\b\in \{0,\ldots, 3\}.
\]
One may regard $\A$ as a $U(1)$ connection and $\F$ as the associated curvature.
We write $\ip{f,g}$ to denote the {\em real} inner product
\[
\ip{f,g}= \mbox{Re} (f\overline g).
\]
In \eqref{lag} we  sum over repeated upper and lower indices, and
we raise and lower indices with the Minkowski metric
$(\eta^{\a\b}) = (\eta_{\a\b}) = \mbox{diag}(-1,1,1,1)$
so that
\[
\ip{D_{\a}\vp,D^{\alpha}\vp} = \eta^{\a\b}\ip{D_{\a}\vp,D_{\b}\vp},
\quad\quad
\F_{\a\b}\F^{\a\b} = 
\eta^{\a\gamma} \eta^{\b \delta} \F_{\a \b}\F_{\gamma\delta}.
\]
We will consider the scaling $0<\e\ll 1$, which is relevant to models describing cosmic strings,
where typically $\e \sim 10^{-16}$ in the units we have (implicitly) chosen. 

We remark that the Lagrangian \eqref{lag}  is invariant under action of the $U(1)$ group, so for any sufficiently smooth function $X:\R^{1+3}\rightarrow \R$ we have
\[
\mathcal L(\vp,\A)=\L( e^{iX} \vp, \A+ dX).
\]

The Euler-Lagrange equations associated to the action functional
$\int_{\R^{1+3}} \L(\vp, \A)$ are
\begin{align}
-D_{\a}D^{\a}\vp  +\frac {\lambda} {4\e^{2}} (\abs{\vp}^{2}-1)\vp=0,
\label{ahm1}\\
-\e^{2}\partial_{\a}\F^{\a\b}-  \eta^{\a\b} \ip{i\vp, D_\a\vp} \ = \ 0.
\label{ahm2}\end{align}
Our main theorem describes the behavior of certain solutions of this system,
for well-chosen initial data.

\subsection{the Nambu-Goto action: timelike minimal surfaces}\label{S:n-g}
The worldsheet of a closed string may be described by a function
$H:(-T,T)\times S^1\rightarrow (-T,T)\times \R^3$
of the form
\begin{equation}
H(y^{0},y^{1})=(y^{0}, h(y^{0}, y^{1})) \quad\qquad\mbox{ for some  $h: (-T,T)\times S^1 \to \R^3$.}
\label{gamma1}\end{equation}
Here and throughout this paper, $S^1$ denotes $\R/L\Z$ for some $L>0$,
so that $S^1$ is a circle of arbitrary positive length $L$.
We write $\Gamma$ for the image of such a map $H$, and the induced metric on $\Gamma$
is denoted by
\[
\gamma_{ab} = \eta_{\a \b} \partial_a H^\a \partial_b H^\beta, \quad  a,b \in \{0,1\},
\]
where we implicitly sum over repeated indices $\alpha, \beta = 0,\dots, 3$.
A surface $\Gamma$ is said to be {\em timelike} if $\det(\gamma_{ab})<0$
at every point in $(-T,T)\times S^1$.
The Nambu-Goto action is proportional to 
\begin{equation}
\mathcal {NG}( H):= \int \sqrt{-\gamma}
\quad\qquad
\ \mbox{ with } \ \gamma : = \det(\gamma_{ab}).
\label{NambuGoto}\end{equation}
A timelike surface $\Gamma = \mbox{Image}(H)$ is called a minimal surface if $H$ is a critical point
of $\mathcal {NG}$.

A timelike minimal surface may be written in
conformal coordinates, in which case 
\begin{equation}
\gamma_{01} = \gamma_{10} = 0, \qquad 
-\gamma_{00} = \gamma_{11}.
\label{conformal}\end{equation}
This is well-known in the physics literature, see for example \cite[Section 6.2]{vs}, and is
proved in \cite{BHNO}. We will always assume  $H$ is a smooth timelike\footnote{
It follows from results in \cite{NT}, \cite{JNO}, or by inspection of  \eqref{our_surface}, that $T < L$ if the image of $h_0'$ contains antipodal points in $S^2$. }  embedding
on $(-T,T)\times S^1$ and that \eqref{gamma1}, \eqref{conformal} hold.
With a conformal parametrization, \eqref{conformal}, $h(y^0,y^1)$ may be written in
the form
$\frac 12(a(y^0+y^1) + b(y^0-y^1))$ for functions $a,b:S^1\to \R^3$ such that $|a'| = |b'| = 1$, and conversely 
every map of this form parametrizes a minimal surface. In particular, if $h_0:S^1\to \R^3$
is an arclength parametrization of a smooth embedded curve $\Gamma_0=\mbox{image}(h_0)$, then
the timelike minimal surface that agrees with $\Gamma_0$ at time $t=0$ and with
zero initial velocity can be written in the form \eqref{gamma1}, with
\be
h(y^0,y^1) = \frac 12 ( h_0(y^1+y^0) + h_0(y^1-y^0)).
\label{our_surface}\ee
This is the situation that we will always consider, although we will rarely need the 
explicit formula \eqref{our_surface}. 

Since $H$ is smooth and $\det \gamma<0$ in $(-T,T)\times S^1$,
it is  clear that for every $T_1<T$, there exists some $c_0 = c_0(T_1)>0$ such that
\begin{equation}
\gamma_{11} = - \gamma_{00} \ \ge c_0
\quad\mbox{ for all $(y^0,y^1)\in (-T_1, T_1)\times S^1$. }
\label{gamma3}
\end{equation}
% Milbredt '08: local in time existence of smooth timelike minimal surfaces.  See also Deck, Gu, M\"uler for results on motion of strings.

\subsection{The Euclidean abelian Higgs model in 2 dimensions}\label{S:2deuc}

%The proposal of Nielsen and Olesen  \cite{NielsenOlesen} is that the abelian Higgs model in $1+3$ dimensions should have solutions $\calU$ that are approximately translation-invariant along a timelike minimal surface $\Gamma$,  and such that the restriction of $\calU$ to every $2d$ slice normal to $\Gamma$  is close to what they call  a {\em vortex solution} of the 2-dimensional Euclidean abelian Higgs model. To describe these, we discuss the 2d Euclidean model. 

We will write the $2d$ abelian Higgs energy density in the form
\be
e^\nu_{\e,\lambda}(U) := 
\frac 12 (|D_1\phi |^2 +|D_2\phi |^2) + \frac {\e^2}4(F_{12})^2 + \frac \lambda{8\e^2} (|\phi |^2-1)^2.
\label{enel}\ee
Here $U = (\phi, A)$, where
$\phi\in H^1_{loc}(\R^2;\C)$ and $A= A_1 dy^1 + A_2 dy^2$ is a $1$-form
with  components in $H^1_{loc}$. We write as usual $F_{12} = \partial_1 A_2 - \partial_2 A_1$,
so that $dA = F_{12} dy^1\wedge dy^2$.
A {\em finite-energy configuration} is a pair $U = (\phi, A)$ such that 
$e^\nu_{\e,\lambda}(U)\in L^1(\R^2)$.

%where here, and frequently throughout this paper, the superscript $\nu$ is a reminder
%that we are considering 
%the $2$-dimensional space of ``normal" variables.

Note that $\e$ is just a scaling parameter in \eqref{enel}, and one can easily change variables to set $\e=1$. 
That is, given a configuration $U = (\phi,A)$, if we define $U^\e=(\phi^\e,A^\e)$ by $\phi^\e(y) := \phi(\frac y \e), A^\e(y) := \frac 1{\e} A(\frac y\e)$, then 
\be
e^\nu_{\e,\lambda}(U^\e)(y) \ = \frac 1{\e^2} \, e^\nu_{1,\lambda}(U)(\frac y \e),\quad\quad
\mbox{ and thus }
\int_{\R^2} e^\nu_{\e,\lambda}(U^\e) \ =  \ 
\int_{\R^2} e^\nu_{1,\lambda}(U) .
\label{scale_e}\ee
However,  we find it convenient to include the scaling parameter $\e$
in the energy.

We define the (2-dimensional) current $j(U)$ and vorticity $\omega(U)$, given by
\begin{align}
j(U) &:= ( \ip{ i\phi, D_1 \phi} ,  \ip{ i\phi, D_2 \phi} ),
\label{j.def}\\
\omega(U &)=\frac 12\nabla \times(j(U) +A)  \ = \ \frac 12 \big[ \partial_1 j_2(U) - \partial_2j_1(U) + F_{12}\big].
\label{omega.def}\end{align}
As we will recall in slightly more detail in Section \ref{S:euclidean2d}, if $U$ is a finite-energy
configuration then
$\omega(U)\in L^1(\R^2)$, and moreover
\be
\int_{\R^2} \omega(U) \ dy \in \pi \Z \quad\quad\mbox{ for every finite energy }U.
\label{quant}\ee
It follows that every finite-energy $U$ belongs to exactly one of the sets
%weak homotopy classes
\be
H_n := \{ U = (\phi, A)\in H^1_{loc}\times H^1_{loc} : \int_{\R^2} e^\nu_{\e,\lambda}(U) < \infty, \int_{\R^2} \omega(U) = \pi n \}.
\label{Hn.def}\ee
(These sets are called  {\em weak homotopy classes} by Rivi\`ere \cite{Riviere02}, who 
establishes a slightly different description of them.)
Note also that, while $\int_{\R^2} e^\nu_{\e,\lambda}(U)$ certainly depends on
$\lambda$, the condition $ \int_{\R^2} e^\nu_{\e,\lambda}(U)<\infty$ is independent of
$\lambda$, and hence the homotopy classes $H_n$ are also  independent of $\e$ and $\lambda$.
We will use the notation 
\be
\calE^\lambda_n := \inf \{ \int_{\R^2} e^\nu_{\e,\lambda}( U) \ : \ U \in H_n\}.
\label{Alambdan.def}\ee
Our main results describe solutions of the $1+3$-dimensional abelian Higgs model
in terms of solutions, when they exist, of the $2d$ minimization problem:
\be
\mbox{ find $U^{m} = U^{m}_{\e,\lambda}\in H_m$ such that }\qquad
\int_{\R^2} e^\nu_{\e,\lambda}(U^{m}) = \calE^\lambda_m.
\label{2dmin}\ee
%This problem depends on $\e$ only through scaling, see \eqref{scale_e}. 
For $0<\e\ll1$, the regime that interests us, we  always assume that a minimizer
$U^m_{\e,\lambda}$ is obtained by starting from a fixed minimizer of the $\e=1$
problem and scaling as in \eqref{scale_e}, so that the energy and vorticity concentrate 
near the origin.

%We summarize some facts, all of which are independent of $\e>0$,about  problem \eqref{2dmin}.

%\begin{itemize}
%\item 

\begin{rem}
For every $n\in \Z$ and $\lambda>0$, there exists  an {\em equivariant} $U^{(n)}\in H_n$
solving the Euler-Lagrange equations associated to the minimization problem \eqref{2dmin}, see \cite{%Plohr,
BergerChen}. Here ``equivariant" implies for example that $\phi^{(n)}$ can be written in the form $f(r)e^{in\theta}$.
The equivariant solution is known to be linearly stable if $|n|= 1$ or $\lambda \le 1$, and linearly unstable (and hence not
even a local energy minimizer) if $\lambda>1$ and $|n|\ge 2$, see Gustafson and Sigal \cite{GSstab}.
\end{rem}

\begin{rem}%\item 
\label{rem_JTconj}
The  conjecture of Jaffe and Taubes \cite[Chapter III.1, Conjectures 1 and 2]{JaffeTaubes} mentioned earlier holds
that the equivariant solution solves problem \eqref{2dmin} for all parameter
values for which it is linearly stable, and that no minimizer exists whenever the equivariant
solution is linearly unstable.
 %\item 
This  is known to be true in the case $\lambda=1$, which has a special structure that
will be recalled in Section \ref{S:euclidean2d}, and for all sufficiently large $\lambda$,
due to work of Rivi\`ere \cite{Riviere02}. Otherwise it is open, as far as we know.
\end{rem}

%\item 

\begin{rem}\label{rem_linear_growth}
In Theorem  \ref{T.exist} we establish a general sufficient  condition, 
%more or less implicit in arguments of \cite{Riviere02}  and 
involving the behavior of the map $m\mapsto \calE^\lambda_m$,  
for existence of solutions of problem \eqref{2dmin}.
In particular we deduce from this that a minimizer exists for $|n|=1$ and $\frac 15 < \lambda < 5$.
%\item Simple scaling considerations show that if an energy-minimizer
%$U^m = U^m_{\e,\lambda}$ exists in $H_m$, one can arrange (after a translation if necessary) that
%\be
%\int_{\R^2} |\omega(U^m)|(y)\  \min( 1, |y|^2) \ dy \le C(\lambda,n) \e^2.
%\label{omega_conc}\ee
%We will always assume that this is the case,
%so that the vorticity of $U^m$ is concentrated near the origin.

%\item 

Theorem \ref{T.exist} implies in particular that if $\lambda$ satisfies
\be
\calE^\lambda_n \ge \calE^\lambda_1\ \ \ \mbox{ for } n >1,
\label{E1minimal}\ee
then problem \eqref{2dmin} has a solution
for $|n|=1$. 
Condition  \eqref{E1minimal} is known to hold
for large $\lambda$, see \cite{Riviere02}, and
we note in Lemma \ref{lineargrowth}  that it is easily
verified for $\frac 12 \le \lambda \le 2$. It is expected
that \eqref{E1minimal} holds for all $\lambda>0$, and
more generally that  $n\mapsto \calE^\lambda_n$ is increasing
for $n\in \mathbb N$. (It is easy to check that 
$\calE^\lambda_{-n} = \calE^\lambda_n$ for all $n$.)
A statement similar to \eqref{E1minimal} is proved
for  certain {\em non-gauged} generalized Ginzburg-Landau-type models
by Almog {\em et al} in \cite{AlmogEtAl}, but adapting their arguments 
to the gauged case seems not to be easy.
\end{rem}

%\end{itemize}

\subsection{Normal coordinates}\label{cov}
Next we describe a useful coordinate system,
which we will refer to as {\em normal coordinates},
for a neighbourhood of a minimal surface $\Gamma$.
A key point in our analysis (as in \cite{Jerrard09}) 
will be to obtain estimates in these coordinates.
%, and  some of the hypotheses and conclusions of Theorem \ref{mainthm} also appear more natural in this coordinate system.

Given a minimal surface $\Gamma$, always assumed to be 
represented via a conformal parametrization $H: (-T,T)\times S^1\to \R^{1+3}$,
see  \eqref{gamma1}, \eqref{conformal},
we will parametrize a neighborhood of $\Gamma$ by
$(-T,T)\times S^1\times \R^2$,
and we will write points in this set as
\be
\mbox{$y = (\yt, \yn)$, \ \  with $\yt = (y^0,y^1) \in (-T,T)\times S^1$ and $\yn = (y^{2}, y^3)\in \R^2$,}
\label{tn.notation}
\ee
where the superscripts stand for ``tangential'' and ``normal'' respectively.
We will also sometimes write 
\be
(y^{\nu1}, y^{\nu2}) = (y^2, y^3).
\label{tn.n1}\ee
We will arrange that  $y^0$ is a timelike coordinate and $y^1,\ldots, y^3$ spacelike. 
%We will also sometimes write  $y' := (y^1,y^2, y^3)$, so that a

To define these coordinates, we first fix maps 
$\bar \nu_i:(-T,T)\times S^1\to  \R^{1+3}, i=1,2$ such that 
\be
\eta_{\a\b} \bar \nu_i^\a  \bar \nu_j^\b = \delta_{ij}, 
\qquad\qquad
\eta_{\a\b} \bar \nu_i^\a  \partial_{0}H^\beta = 
\eta_{\a\b} \bar \nu_i^\a  \partial_{1}H^\beta = 0
\label{barnu}\ee
for $i,j = 1,2$.
In other words, $\{ \bar \nu_i(\yt)\}_{\i=1}^2$
is an orthonormal frame (with respect to the Minkowski metric) 
for the normal bundle to $\Gamma$
at $H(\yt)$.
We then define $\psi: (-T,T)\times S^{1}\times \R^2 \rightarrow\R^{1+3}$
by
%extend it to a diffeomorphism $\psi$ of a neighborhood in $(-T,T)\times S^{1}\times \R^{2}$ of $(-T,T)\times S^{1}$  and a neighborhood of $\Gamma$ in $ (-T,T)\times \R^{3}$.  Define $\psi$ as follows
\be%&\psi: (-T,T)\times S^{1}\times B_\nu(\rho_{0}) \rightarrow (-T,T)\times \R^{3},\\
\psi(y)=H(\yt)+\bar{\nu_{1}}(\yt)y^{2}+\bar{\nu_{2}}(\yt)y^{3}  \ = \ 
H(\yt)+\bar{\nu_{1}}(\yt)y^{\nu1}+\bar{\nu_{2}}(\yt)y^{\nu2} .
\label{psi.def}\ee
Writing $B_\nu(\rho) :=  \{ \yn = (y^{\nu 1},y^{\nu2}) \ : |\yn|<\rho\}$, we will  restrict $\psi$ to
a set of the form $(-T_1,T_1)\times S^1\times B_\nu(\rho_0)$ on which $\psi$ is injective
and satisfies other useful properties; see Section \ref{ahm} for details.

Since our argument will rely heavily on specific properties of the abelian Higgs
model when written with respect to the new coordinates, we find it useful
to distinguish between the Higgs field and connection when written
in  terms of the original, standard coordinates for Minkowski spacetime,
which we will write $(\varphi, \A)$, and the 
same objects written in terms of the new coordinates, which
we will denote $(\phi, A)$. 
These are related by
\be
\phi = \varphi \circ \psi,\quad\quad\quad A = \psi^* \A,\quad\mbox{ so that }
A_\alpha = \partial_\a\psi^\beta\,  \A_\beta \circ \psi
\label{cvnotation}\ee
on $\mbox{domain}(\phi) = (-T_1,T_1)\times S^1\times B_\nu(\rho_0)$.
The components of the curvature in the two coordinate systems will be
denoted $\F_{\a\b}$ and $F_{\a\b}$ respectively. We will also write 
$\calU$ to denote a pair $(\vp, \A)$ and similarly $U$ for a pair $(\phi, A)$,
and we will write 
\be
U = \psi^* \calU\quad\quad\mbox { when $\calU$ and $U$ are related as in \eqref{cvnotation}.}
\label{cvn1}\ee
We will also write $\calU = (\psi^{-1})^* U$ to indicate that \eqref{cvnotation} holds.

\subsection{main theorem}

Our main result, stated below, asserts the existence of a solution whose energy concentrates around a minimal surface $\Gamma$, and that in a neighborhood of
$\Gamma$ is close to a configuration 
that in the $y$ coordinates takes the form
$U^{\mbox{\scriptsize N\!O}}=(\phi^{\mbox{\scriptsize N\!O}}, A^{\mbox{\scriptsize N\!O}})$,
with
\be\label{N-O}
\phi^{\mbox{\scriptsize N\!O}}(\yt, \yn) := 
\phi^m(\yn),
\qquad
A^{\mbox{\scriptsize N\!O}}(\yt, \yn)
:=  A^{m}_1(\yn) dy^{\nu1} +  A^{m}_2(\yn) dy^{\nu2},
\ee
where $U^m = (\phi^m, A^m) = U^m_{\e,\lambda}$ is a ground state of the $2d$ minimization
problem \eqref{2dmin}.
Thus, in standard coordinates this configuration can be written
\be
\calU^{\mbox{\scriptsize N\!O}} =(\varphi^{\mbox{\scriptsize N\!O}}, \A^{\mbox{\scriptsize N\!O}})
= (\psi^{-1})^* U^{\mbox{\scriptsize N\!O}}.
\label{form1}\ee
Note that $\calU^{\mbox{\scriptsize N\!O}}$ is only defined in 
the domain of $\psi^{-1}$, which is a neighborhood of $\Gamma$.

\begin{thm}\label{mainthm}
Let $\Gamma$ be a codimension $2$ timelike minimal surface,
given as the image of  a conformal parametrization $H$ (so that $H$ satisfies \eqref{gamma1}, \eqref{conformal}) that is a smooth embedding in $(-T,T)\times S^1$.  
Assume also that the initial velocity of  $\Gamma$ at $t=0$ is everywhere $0$.
 
\medskip
 
Let  $\lambda>0$  and $m \in \Z$ be such that the 2d minimization
problem \eqref{2dmin} has a solution,  and  in addition
assume that $\calE^\lambda_m \le \calE^\lambda_n$ whenever $|m| \le |n| $.

\medskip

Then, given $T_0 <T$, there exists an neighborhood $\calN_1\subset \mbox{Image}(\psi)$
of $\Gamma$ in $(-T_0,T_0)\times \R^3$, and a constant $C$, both
%of which may depend on $\Gamma$ and $T_0$ but are 
independent of $\e$,
such that  given $\e\in (0,1]$, there exists a solution
$\calU$ of the abelian Higgs model \eqref{ahm1}, \eqref{ahm2}
satisfying the following estimates. First,
\be
\int_{\calN_1} | \vp - \vp^{\mbox{\scriptsize N\!O}}|^2 + \e^2 |\A - \A^{\mbox{\scriptsize N\!O}}|^2 \le C \e^2
\label{T1.c1}\ee
in a suitable gauge, for $\calU^{\mbox{\scriptsize N\!O}}$ defined in \eqref{form1}. Second,%\be
%\int_{\calN } | \calU_\e - \calU^{\mbox{\scriptsize N\!O}} |^2 dx dt \leq C\e,
%\label{T1.c1}\ee
\be
\int_{ \calN_1}  (d^\nu)^2
\left( |D\varphi|^2 + \e^2 |\F|^2 + \frac \lambda{8\e^2}(|\varphi|^2-1)^2 \right) \ dx \ dt \ \le C \e^2,
\label{T1.c3}\ee
where $d^\nu:\calN_1\to \R$ is the distance in normal coordinates to $\Gamma$,
so that
$d^\nu\circ \psi(y) := |\yn|$. And finally,
\be
\int_{\left[(-T_1,T_1)\times \R^3\right]\setminus \calN_1}  |D\varphi|^2 + \e^2 |\F|^2 + \frac \lambda{8\e^2}(|\varphi|^2-1)^2  \ dx \ dt \ \le C \e^2.
\label{T1.c2}\ee
\end{thm}
\begin{rem}
From the construction of the initial data and conservation of energy, we will have $\int_{ \{ t \}\times \R^3}|D\varphi|^2 + \e^2|\F|^2 \ge C >0$ for
every $t$. Thus \eqref{T1.c2}, \eqref{T1.c3} contain highly nontrivial information about energy
concentration around $\Gamma = \{ d^\nu = 0\}$.
\end{rem}
\begin{rem}
In fact we prove a more general stability result, giving
estimates for any solution that at $t=0$ has a vortex filament near $\Gamma_0$
and satisfies certain smallness conditions. For details see Proposition \ref{mainprop2},
from which one can also extract further estimates satisfied by the particular solution 
$\calU$ of Theorem \ref{mainthm}.
\end{rem}

\begin{rem}The hypotheses on $\lambda$ and $m$ are known to be satisfied for
\begin{itemize}
\item  $|m|=1$
and $\lambda\in [ \frac 12, 2]$. This follows from Lemma \ref{lineargrowth} and Theorem \ref{T.exist} below.
\item $|m|=1$ and all $\lambda$ larger than some $\lambda_0$, see \cite{Riviere02}.
\item $\lambda = 1$ and all $m\in \Z$, see \cite{JaffeTaubes}.
\item any $\lambda>0$, and $m$ minimizing $n\mapsto \calE^\lambda_n$ among
nonzero integers. This again follows from Theorem \ref{T.exist}. 
\end{itemize}
They are believed to hold for all $\lambda>0$ when $|m|=1$, and for all $m\in \Z$ when $\lambda\in (0,1)$.  (See the next remark and remarks \ref{rem_JTconj} and \ref{rem_linear_growth}) .%This is essentially a conjecture of Jaffe and Taubes, mentioned above.
\label{rem.mlambda}\end{rem}

\begin{rem}Nielsen and Olesen \cite{NielsenOlesen} and authors in the
subsequent physics literature have in mind solutions of the form
$\calU \approx \calU^{\mbox{\scriptsize N\!O}}$, where $\calU^{\mbox{\scriptsize N\!O}}$ satisfies
\eqref{N-O}, and  $U^m$
is an equivariant $m=1$ solution of the $2d$ abelian Higgs model.  
Thus, their exact scenario is established in
Theorem \ref{mainthm} 
in the cases ($\lambda = 1$ or $\lambda$ large) 
when the $m=1$ minimizer is known to (exist and)
coincide with the equivariant solution. A full proof of their
conjecture would involve showing that the $m=1$ 
equivariant solution
is  a minimizer of the 2d Euclidean energy 
in the weak homotopy class $H_1$ for every $\lambda$.
This is exactly the $m=1$ case of the conjecture of Jaffe and Taubes
mentioned in Remark \ref{rem_JTconj}.
\label{rem:N-O}\end{rem}

\begin{rem}The estimates obtained in \cite[Theorem 2]{Jerrard09} for the non-gauged analog of the
abelian Higgs model are similar to \eqref{T1.c3}, \eqref{T1.c2} but {\em much} weaker.
Indeed, they show that the total energy diverges  like $|\ln \e|$, whereas the weighted energy
(as in  \eqref{T1.c3}, \eqref{T1.c2}) is bounded as $\e\to 0$. Thus energy concentrates very
weakly around the manifold $\Gamma$. In addition, no useful estimate along the
lines of \eqref{T1.c1} is obtained in \cite{Jerrard09}.
% One difficulty in the non-gauged case is that, 
%contrary to Remark \ref{r:canonical}, the analog of $\calU^{\mbox{\scriptsize N\!O}}$
%depends in an essential way on the details of the construction of the diffeomorphism
%$\psi$, and to get a good analog of \eqref{T1.c1} it would presumably be necessary 
%to identify a nearly-optimal $\psi$. 
\end{rem}

\begin{rem}  The assumption on $\Gamma$ to have zero initial velocity is there to avoid some technicalities of \cite{Jerrard09}.  We also use it in the construction of the initial data.  In principle, one could repeat the steps and assume nonzero velocity.  
\end{rem}

\subsection{global well-posedness for the abelian Higgs model}\label{s:gwp}
The abelian Higgs model is a $U(1)$ version of the Yang-Mills-Higgs (YMH), where the gauge group is in general nonabelian.  For the purposes of this article, we need the $1+3$ dimensional abelian Higgs model to be well posed for $H^1_{loc}\times H^1_{loc}$ data (this is made precise in Section \ref{s:id} below).  In addition, since we are going to rescale, we need the well-posedness for large data and for all time (we would like the analysis to hold at least for the existence of the time-like minimal surface).  Finally, we are interested in the topological behavior at infinity, $\abs{\varphi}\rightarrow 1$ instead of having $\abs{\varphi}\rightarrow 0$.  
\newline\indent  The strongest well-posedness result in the literature for YMH is due to Keel \cite{Keel}.  It shows global well-posedness of the $1+3$ solution in the energy class for any size data in the temporal gauge.  Moreover, the Higgs potential is taken to be energy critical, $V(\varphi)=\abs{\varphi}^p$ with $p=6$; the power six is the highest power that can be controlled using the Sobolev embedding by the kinetic part of the energy.   A power $p<6$ is called subcritical, and is significantly easier to handle.  Since we have a quartic potential, the global well-posedness for the abelian Higgs model we need, in the temporal gauge, is implied by \cite{Keel}.  The only detail left is then addressing $\abs{\phi}\rightarrow 1$ (see Section \ref{s:id} below). 
\newline\indent On the other hand, the proof in \cite{Keel} is more sophisticated than what we need, and not only because of the critical power of the potential.  An intermediate step leading to the global result is changing to a Coulomb gauge.  In the nonabelian case, the Coulomb gauge can be constructed only locally in space, and hence the nonabelian case is much more technical than the abelian one, where the global Coulomb gauge can be constructed.  
\newline\indent  Therefore, due to having subcritical potential and the abelian problem, heuristically speaking, the global well-posedness we need can also follow from the work done on the Maxwell-Klein-Gordon problem \cite{KM, ST}.  
%Maxwell-Klein-Gordon system is the abelian Higgs model without the potential.  
%Klainerman and Machedon obtained global well-posedness in the Coulomb gauge \cite{KM}.  The Lorenz gauge was handled recently by Selberg and Tesfahun \cite{ST}. 
\subsection{initial data}\label{s:id}
% By the global in time solution in the energy class we mean a pair  $(\varphi, \mathcal A)$ such that $(\varphi, \mathcal A)$ solves \eqref{ahm1}-\eqref{ahm2} and
% \be\label{energy_class}
% \vp, \A \in C([0,\infty); H^1(\R^3))\cap C^1([0,\infty); L^2 (\R^3)).
% \ee
The solution $\calU$ that we find in Theorem \ref{mainthm} will be obtained by invoking
the results in \cite[Theorem 1.2]{Keel}. To do this we will impose the temporal gauge
\be
\mathcal A_0(0)=0\label{id3},
\ee
and we require initial data such that
\begin{align}
\vp(0), \mathcal A_i(0) &\in H^1_{loc}(\R^3)\label{id1},\\
\partial_t \varphi(0), \partial_t \mathcal A_i(0)&\in L^2_{loc}(\R^3), \ i=1, 2, 3,\label{id2}
\end{align}
with the compatibility condition (stated in the temporal gauge)
 \be\label{comp}
 \e^{2}\partial_i \partial_t\A_i(0)+\ip{i\vp(0),\partial_{t}\vp(0)}=0.
\ee
Note in particular that \eqref{id2}, \eqref{comp} hold if 
$\partial_t \varphi(0) = \partial_t \mathcal A_i(0)=0, \ i=1, 2, 3$, which will be the
case for us.
The initial data $\calU|_{t=0}$
is carefully constructed in the proof of Theorem \ref{mainthm} in Section 7. 
It has a rather explicit description near $\Gamma \cap \{t=0\}$,
in terms of the minimizer $U^m$ from \eqref{2dmin} and the diffeomorphism
$\psi$ from \eqref{psi.def}, and away from $\Gamma \cap \{t=0\}$ it has the form
 \be\label{flat_vacuum}
\vp(0) = e^{i q}, \qquad \A_i(0) = \frac{\partial q}{\partial x^i},\quad i=1, 2, 3.
\ee
for some smooth $q$. %, corresponding to a flat vacuum. 
From these facts it follows that \eqref{id1} holds, and 
from  \cite{Keel} we then obtain a global solution in the temporal gauge.  

We note that in particular we will consider data such that
\eqref{flat_vacuum} holds in $\R^3\setminus B_R$ for some $R$. By
finite propagation speed and an easy explicit calculation,  
$\vp(t) = e^{iq}, \A(t) = dq$ is a solution on $|x|>R+t$.  By uniqueness, it must agree with the solution obtained using \cite{Keel}. 

%\blue{ i think there is almost nothing to check here. the point is this: fix a time $T$. 
%Fix a large ball $B_R$ such that the initial data is a flat vacuum ($\vp = e^{iq},
%\A = dq$, and in the temporal gauge, $\partial_t \vp = 0, \partial_t \A = 0$)
%outside of $B_{R-T}$. Then the solution on $[0,T]\times B_R$ will be the same
%$\vp = e^{iq}, \A = dq$, by finite propagation speed and since it is immediate that this is
%a solution. Then apply local-in-space well-posedness results on 
%the cone $\{ (x,t) : |x|  < R+T-t\}$. Or if you like, apply them on the cylinder $B_R\times [0,T]$ with
%dirichlet data $(\vp, \A) := e^{iq}, dq)$ on $\partial B_R\times [0,T]$. 
%Piecing these together we get a solution on
%$\R\times [0,T]$. Since $T$ was arbitrary we have global existence and uniqueness.}

\subsection{some notation}
As mentioned above, we implicitly sum over repeated upper and lower indices.
We  use the convention that greek indices $\alpha,\beta, \mu,\nu...$ run from $0$ to $3$, and latin indices $i, j, k....$ run from $1$ to $3$.  

For the convenience of the reader, we include the following summary of the different solutions we work with
\begin{itemize}
\item In normal coordinates: $U=(\phi, A)$ 
\item In the standard Minkowski coordinates: $\mathcal U=(\varphi, \mathcal A)$, and when applicable we use \eqref{cvn1} to relate $U$ and $\mathcal U$.
\item $U^{m}=U^m_{\epsilon, \lambda}$ solution of the 2D minimization problem \eqref{2dmin}
\item $U^{(m)}$ \emph{equivariant} solution of the 2D minimization problem 
\item $U^{\mbox{\scriptsize N\!O}}$ Nielsen-Olesen solution in the normal coordinates given by \eqref{N-O}
\item $\mathcal U^{\mbox{\scriptsize N\!O}}$ Nielsen-Olesen solution in the standard Minkowski coordinates given by \eqref{form1}
\end{itemize} 
\subsection{organization of this paper}
%As mentioned in the introduction, we follow general framework developed in \cite{Jerrard09}. 
Sections \ref{S:euclidean2d} -  \ref{minimizer} deal with aspects of the 2d Euclidean abelian Higgs model needed for our main dynamical results. We start in Section \ref{S:euclidean2d} with some general background material. Section \ref{S:Dnu} introduces, 
and establishes some basic properties of,
what we call a {\em vorticity confinement functional}. This functional plays an important
role in the proof of Theorem \ref{mainthm}.
In Section \ref{minimizer} we prove Theorem \ref{T.exist},
giving a criterion for existence of solutions of the minimization problem \eqref{2dmin}.
As mentioned in Remark \ref{rem.mlambda} above, this result show that the hypotheses of
%This result is not exactly used anywhere  in this paper, but
%as mentioned in Remark \ref{rem.mlambda} above, it shows that the hypotheses of
Theorem \ref{mainthm} are satisfied for a range of values of the parameters $m,\lambda$.
 
Sections \ref{ahm} and \ref{proof_mainthm} consider 
the abelian Higgs model in $1+3$-dimensional Minkowski space.
A basic ingredient in our analysis, as in \cite{Jerrard09}, is supplied by weighted
energy estimates in the normal coordinate system, introduced in
Section \ref{cov}. These estimates are proved in Section \ref{ahm},
using results about the vorticity confinement functional
from Section \ref{S:Dnu}. 
Finally, section \ref{proof_mainthm} is devoted to the proof of Theorem \ref{mainthm}.

\section{Energy and vorticity in 2 dimensions}\label{S:euclidean2d}

In the next three sections, we focus on Euclidean abelian Higgs model in 2 dimensions.
In this section we record some  facts, mostly well-known, relating the energy 
$e^\nu_{\e,\lambda}$ and the vorticity $\omega$, defined in \eqref{enel} and \eqref{omega.def} respectively.
%These will be used through the rest of the paper.
We recall that the parameter $\e$ is just a scaling parameter, see \eqref{scale_e},
so that all results in this section reduce to the case $\e=1$. 
However, due to the role it plays elsewhere in this paper, it seems
useful to formulate things here for general $\e>0$.

First, by a direct computation we have the following identity, due to Bogomol'nyi:
% (see for example \cite{JaffeTaubes}) which in our notation takes the form
%\begin{align*}
%&\frac 12 \abs{(D_{1}\pm iD_{2})u}^{2}+\frac 12 \big(\e F_{12}\pm\frac 1{2\e}(\abs{u}^{2}-1)\big)^{2}\pm \frac 12 F_{12}+\frac{\lambda-1}{8\e^2}(\abs{u}^{2}-1)^{2}\\
%&\quad=\frac12 \abs{D_{A}u}^{2}+\frac {\e^2}2\abs{F_{12}}^{2}+\frac{\lambda}{8\e^2}(\abs{u}^{2}-1)^{2}\mp \frac 12\nabla \times j,
%\end{align*}
%or if we let
%\[
%e(\phi, A)=\frac12 \abs{D_{A}u}^{2}+\frac 12\abs{F_{12}}^{2}+\frac{\lambda}{8}(\abs{u}^{2}-1)^{2},
%\]
%and it follows that 
\be\label{Btrick}
e^\nu_{\e,\lambda}(U)=\pm \omega(U) + \frac 12 \abs{(D_{1}\pm iD_{2})\phi}^{2}+\frac 12 \big (\e F_{12}\pm\frac 1{2\e}(\abs{\phi}^{2}-1)\big)^{2}+\frac{\lambda-1}{8\e^2}(\abs{\phi}^{2}-1)^{2}.
\ee
%Indeed, this follows by expanding the squares, from noting that 
%\[
%2\langle D_1 u, iD_2 u\rangle 
%= 
%-\langle iD_1 u, D_2 u\rangle +
%\langle   iD_2 u,   D_1 u\rangle 
%=
%-\partial_1 
%\langle  iu, D_2 u\rangle +
%\partial_2 
%\langle  iu, D_1 u\rangle 
%+ \langle iu,[D_1,D_2]u  \rangle
%\]
%and using the commutator relation.
We emphasize that the identity holds pointwise. 
Note that \eqref{Btrick} implies that
\be
|\omega(U)| \le  \max\{ 1, \lambda^{-1}\} \ e_{\e,\lambda}^\nu(U)
\label{Bogcor}\ee
pointwise.
This follows immediately from \eqref{Btrick} if
$\lambda \ge 1$, and if $\lambda \le 1$ it follows by noting that 
$|\omega| \le e_{\e, 1}^\nu(U) \le \lambda^{-1}e^\nu_{\e,\lambda}(U)$.

We immediately deduce from \eqref{Bogcor} that $\omega(U)$ is integrable
for any finite-action $U = (\phi, A)$, and it is known  (and follows rather easily
from Lemma \ref{lemma2} below) that $\int_{\R^2}\omega(U) \in \pi\Z$.

Let $U^{(m)}$ denote the $\lambda=1$, 
equivariant solution in the weak homotopy class
$H_m$, discussed in Section \ref{S:2deuc}.
It is well-known that $U^{(m)}$ satisfies
\[
(D_{1}  + \sigma iD_{2})\phi = 0, 
\qquad
F_{12} + \sigma \frac 1{2\e}(\abs{\phi}^{2}-1) = 0,
\qquad\qquad \sigma:= \operatorname{sign}(m),
\]
see for example \cite{JaffeTaubes}.
By combining these 
with \eqref{Btrick}
we see that
\be
\calE^1_m = \pi |m| = \int_{\R^2} | \omega(U^{(m)}) | .
\label{A1}\ee
From this we easily deduce the following
%Therefore we can write
%\[
%e(\phi, A)\geq \pm (\nabla \times j+F_{12})+\frac{\lambda-1}{8}(\abs{u}^{2}-1)^{2}
%\]
%%does this or something similar hold in higher d???
%We note the identity holds pointwise.  

\begin{lem}\label{lineargrowth}
If $ \frac {|m|}{|m|+1} \le \lambda \le  \frac {|m|+1}{|m|}$, then 
\be
\mbox{$\calE_{m}^{\lambda}<\calE_{n}^{\lambda}$ whenever $|m| < |n|$. }
\label{Amonotone}\ee
In particular, $\calE_1^\lambda < \calE_n^\lambda$ for all $|n|\ge 2$ if $\frac 12 \le \lambda \le 2$.
\end{lem}

\begin{proof}
It is clear from \eqref{A1} that the lemma holds  for $\lambda=1$.
If $1 < \lambda \le  \frac {|m|+1}{|m|}$,
then $e^\nu_{\e,1}(U) \le e^\nu_{\e,\lambda}(U) \le \lambda e^\nu_{\e,1}(U)$ pointwise for every $U$. 
In addition,  $\lambda |m| \le |n|$ if $|n|>|m|$. From these we
deduce that
\[
 \calE_m^{\lambda}
\ \le  \ 
\lambda\, \calE^1_m \ = \lambda |m| \pi \  \le \   |n|\pi
\ =  \ \calE^1_n  \le \calE_n^\lambda \quad\quad \mbox{ if }|m|< |n|,
\]
and it is not hard to check that at least one inequality is strict.
If $\frac {|m|}{|m|+1}  \le \lambda < 1$ and $|m|< |n|$, then
$\lambda e^\nu_{\e,1}(U) \le e^\nu_{\e,\lambda}(U) \le e^\nu_{\e,1}(U)$ and
$ |m| \le \lambda |n|$, and the conclusion follows very much as above.
\end{proof}

\begin{rem} With a little more work one can prove by similar arguments
that \eqref{Amonotone} holds for a slightly larger range
of $\lambda$, but these sorts of simple arguments have no hope of proving the
natural conjecture, which is that it is valid for all $\lambda>0$.

We also remark that it is known from \cite{Riviere02}  that  if $\lambda$ is sufficiently large
then \eqref{Amonotone} is true for
all $m$ and $n$.
\end{rem}

We conclude this section by proving the lemma mentioned above, which shows that
the vorticity is approximately quantized on a set on which the boundary energy is not too large. For this we need

\begin{lem}\label{geq.delta}
There exists constant $C$ such that
if $S\subset \R^2$ is a bounded, connected, and simply connected set, and
$\partial S$ is Lipschitz with $|\partial S| \ge \e$,
%with Lipschitz boundary of length at least $\e$, 
and if $\rho$ is a smooth nonnegative function on a neighborhood of $\partial S$, then
\be
\int_{\partial S}\frac 12 \abs{\nabla_{\tau}\rho}^{2}+\frac{\lambda}{8\e^{2}}(1-\rho^{2})^{2}d\mathcal H^{1}\geq \frac {\sqrt\lambda}{C\e}\| 1 - \rho\|_{L^\infty(\partial S)}^{2},
\label{gdc}\ee
where $\nabla_\tau $ denotes the tangential derivative along $\partial S$.
\end{lem}

This is proved in  \cite[lemma 2.3]{Jerrard99} when $S$ is a ball, and 
exactly the same argument applies here, since the proof only involves integrating along $\partial S$, which is isometric to a circle.

\begin{lem}\label{lemma2} 
Assume that $S\subset \R^2$ is connected and simply connected with Lipschitz boundary.
Let $\lambda > 0$.
There exists a constant $C$, depending on $\lambda$, such that if $|\partial S|\geq \e$ then
\[
\abs{\int_{S}\omega-\pi n}\leq C\e \int_{\partial S}[\frac{\abs{\nabla_{A}\phi}^{2}}{2}+\frac{\lambda(\abs{\phi}^{2}-1)^{2}}{8\e^{2}}]d\mathcal H^{1}
\]
for some $n\in \mathbb Z$. Moreover, if $\int_{\partial S} e^\nu_{\e,\lambda}(U) \le \frac 1{C\e}$, then in fact
$n = \deg(\frac{\phi}{|\phi|};\partial S)$.
\end{lem}

\begin{proof}
{\bf Case 1}: If $\inf_{\partial S}|\phi| < \frac 12$, then since $|\nabla_\tau |\phi|| \le |\nabla_A\phi|$, we
can apply Lemma \ref{geq.delta} to $\rho = |\phi|$ to find that 
\[ %\be\label{star}
\e \int_{\partial S}[\frac{\abs{\nabla_{A}\phi}^{2}}{2}+\frac{\lambda(\abs{\phi}^{2}-1)^{2}}{8\e^{2}}]d\mathcal H^{1} \ge  \frac {\sqrt \lambda}{4C}.
\] %\ee
%But clearly 
%\[
%\inf_{n\in \Z}
%\abs{\int_S \omega -\pi n}\leq \frac\pi 2.
%\]
Since 
$\min_{n\in \Z} \abs{a-\pi n}\leq \frac\pi 2$ for every $a\in \R$, this implies \eqref{gdc}.

{\bf Case 2}: We assume that
\be
\abs{\phi} \ \ge \frac 12\ \ \mbox{on}\ \ \partial S.
\label{gtrhalf}\ee
Note in particular that this occurs if $\int_{\partial S} e^\nu_{\e,\lambda}(U) \le \frac 1{C\e}$, due to Lemma \ref{geq.delta}. Because of \eqref{gtrhalf}
we can then write  $\phi = \rho e^{i\eta}$ on $\partial S$,  and in this notation,
\[
j(U)=\rho^{2}(\nabla\eta-A)\ \ \mbox{on}\ \ \partial S,
\]
so that $j+A = \nabla\eta + (\rho^2-1)(\nabla\eta -A)$ on $\partial S$.
Hence
\begin{align*}
\int_{S}\omega&=\frac 12\int_{S}\nabla \times (j+A)\\
			   &=\frac 12\int_{\partial S}(j+A)\cdot \tau\\
			%   &=\frac 12\int_{\partial S}\left[\rho^{2}(\nabla\eta-A)+A\right]\cdot \tau\\
			   &=\frac 12\int_{\partial S}\nabla \eta\cdot \tau+\frac 12\int_{\partial S}(\rho^{2}-1)(\nabla\eta-A)\cdot \tau.
\end{align*}
And the conclusion now follows by noting that
\[
\int_{\partial S}\nabla \eta\cdot \tau  \ = \ 2\pi \deg(\frac \phi{|\phi|}; \partial S) \in 2\pi \mathbb Z,
\]
and, recalling \eqref{gtrhalf},
\begin{align*}
\abs{\frac 12\int_{\partial S}(\rho^{2}-1)(\nabla \eta-A)\cdot \tau}&\leq \frac 12\int_{\partial S}\abs{\frac{\rho^{2}-1}{\rho}}\abs{\rho(\nabla\eta-A)}\\
&\leq
\int_{\partial S}\abs{\abs{\phi}^{2}-1}\abs{\nabla_{A}\phi}\\
%&=2\int_{\partial S}\sqrt{\frac{2\e}{\lambda^{\frac 12}}}\abs{\nabla_{A}u}\abs{\abs{u}^{2}-1}\sqrt{\frac{\lambda^{\frac 12}}{2\e}}\\
&\leq 
% \int_{\partial S}\frac{\e}{\sqrt{\lambda}} \abs{\nabla_{A}u}^{2}+\frac {\sqrt{\lambda}}{4\e}(\abs{u}^{2}-1)^{2}\\
%&= 
\frac{2\e}{\lambda^{1/2}}\int_{\partial S}\frac { \abs{\nabla_{A}\phi}^{2}}{2}+\frac {\lambda}{8\e^{2}}(\abs{\phi}^{2}-1)^{2}.
 \end{align*}
\end{proof}

\section{2d Vorticity confinement functional}\label{S:Dnu}

Let $m$ be a positive integer.
For a configuration $U = (\phi, A)$ on $B_\nu(R)\subset \R^2$, 
we define\footnote{In this section, we write $\yn$ to denote a 
point in $(y^1, y^2)\in \R^2$. This variable plays the same role as
the $\yn$ 
in Sections \ref{ahm} and \ref{proof_mainthm}, where however $\yn = (y^2, y^3)$.}
\be\label{defect}
\D^{\nu}_m(U;R) := \pi m - \int_{B_\nu(R)}f(|\yn| )\,\omega(U)(\yn) \,d\yn
\ee
%\red{NOTATION: in this section: $y$ or $\yn$?)}
where $f: [0,R]\rightarrow [0,1 ]$ is a fixed smooth function satisfying
\be
f(0)=1,
\qquad
f(R)=0, 
\qquad
0 \ge f'(r) \ge -C r^2 \mbox{ for all $r$}
\label{f.def}\ee
where of course $C$ depends on $R$. 
We expect $\D^\nu_m(U;R)$ to be small (or negative) if (at least) $m$ quanta of vorticity
are concentrated near the center of the ball $B_\nu(R)$. 

The main results of this section are the following two propositions, both of which 
relate $e_{\e,\lambda}^\nu$ and $D^\nu_m$. 
They together yield stability properties that are used in a crucial way in our proof of Theorem 
\ref{mainthm}.

Our first proposition will allow us to 
control changes in $\D^\nu_m$. In its statement and proof, we write
points in $(0,T)\times B_\nu(R) $ in the form  $(y^0, \yn)$.

\begin{prop}\label{prop2}
Let $U = (\phi, A)$ be a configuration on
$(0,T)\times B_\nu(R) $, for some  $T>0$ and $B_\nu(R)\subset \R^2$ ,
so that  $A$ has 
components $A_i \in H^1(   (0,T)\times B_\nu(R))$ for $i= 0,1,2$.
Then for every $\lambda>0$, integer $m$, there exists a constant
$C = C(\lambda, m,R)$ such that 
\begin{align}
&\abs{\D_m^{\nu}(U(t);R)-\D_m^{\nu}(U(0);R)}\nonumber \\
&\qquad\qquad
\leq C \int_{(0,t)\times B_{\nu}(R)} \left(\abs{D_{0}\phi}^{2}
+\e^{2}(F_{01}^{2} + F_{02}^2) +\abs{\yn}^{2}e^\nu_{\e,\lambda}(U)\right) \ dy^{\nu}d y^0
\label{changeD}\end{align}
for every $t\in (0,T)$. %, where $\red{\yn = (y^1,y^2)}\in B_\nu(R)$. \red{$\yn$ is more consistent with the rest of the file, but then $=(y^{2},y^{3})$?}
\end{prop}

We will in fact prove something stronger than \eqref{changeD}, but here we have
recorded only the conclusion that is needed for the proof of Theorem \ref{mainthm}.

The other main result of the section shows that control over $\D^\nu_m$ implies
good lower energy bounds.

\begin{prop}Suppose that 
$\lambda>0$ and $m\in \mathbb N$ satisfy
\be
\calE^\lambda_m \le \calE^\lambda_n\quad\quad\mbox{ whenever } n\ge m.
\label{monotone}\ee
Then for every $R>0$, there exist constants $\kappa_1$ and $C$, depending on
$R,\lambda, m$,
such that  if $U = (\phi, A)$ is a finite-energy
configuration satisfying
%= \kappa_1(\Lambda)$  such that if
\[
\D^{\nu}_m(U;R)<\kappa_{1},
\]
then
\be
\int_{B_\nu(R)}e^{\nu}_{\e,\lambda}(U) \ \geq \  \calE^\lambda_m- C\e^2
\label{p1.conc}\ee
for all $\e\in (0,1]$. %In particular this holds for $m=1$ and $\frac 12 \le \lambda \le 2$.
\label{prop1}\end{prop}

We remark that, although one could extract from our arguments estimates of how 
various constants depend on
$\lambda$, we have not made any effort to optimize this dependence, and indeed we 
appeal several times to \eqref{Bogcor}, 
which is far from sharp when $0<\lambda\ll 1$ or $\lambda\gg 1$.

%If $D_{\nu}(U;R)\leq \kappa_{1}$, then by definition of $D_{\nu}(U;R)$ and \eqref{deriv}

\subsection{Proof of Proposition \ref{prop2}}

We will use the following lemma
\begin{lem}\label{lemma3}
Assume the hypotheses of Proposition \ref{prop2}. Then for every $r\in (0,R)$ and $t\in (0,T)$
\be
\abs{\int_{B_\nu(r)}\omega(U(t)) -\int_{B_\nu(r)}\omega(U(0))} 
\ \le  
\ \max\{1,\lambda^{-1}\}   \int_{(0,t)\times \partial B_\nu(r)}e^{3d}_{\e,\lambda}(U)  \ d\mathcal H^2
\label{L3.c}\ee
where $\mathcal H^2$ is $2$-dimensional Hausdorff measure and  $e^{3d}_{\e,\lambda}(U)$ denotes the $3$-dimensional energy
density 
\[
e^{3d}_{\e,\lambda}(U) :=  \frac 12 \sum_{a=0}^2  |D_a \phi|^2 +  \frac{\e^2}2 \sum_{0\le a<b \le 2}F_{ab}^2
+ \frac \lambda{8\e^2}(|\phi|^2-1)^2.
\]
\end{lem}

Note that although $y^0$ is naturally identified with
time, we think of and write $e^{3d}_{\e,\lambda}$  as $3$-dimensional rather than $1+2$ dimensional, since $\frac 12 \sum_{a=0}^2  |D_a \phi|^2$ is a sort of Euclidean (rather than Minkowski) norm-squared of the covariant derivative.

\begin{proof}
For this proof only, we will write
%$\omega_{12} = \partial_1 (j+A)_2 - \partial_2 (j+A)_1$ for the $2$-dimensional vorticity (elsewhere denoted simply $\omega$), and we write 
$\omega$ to denote the $3$-dimensional vorticity, which we identify with the
$2$-form $d(j+A)$ in $B_\nu(r)\times [0,T]$.  Here $j = \sum_{a=0}^2 \ip{ i\phi, D_a\phi} dy^a$.
Since $\omega$ is exact, 
\[
\int_{ \partial( (0,t) \times B_\nu(r) ) }\omega = 
\int_{ (0,t) \times B_\nu(r)  } d\omega 
=  0.
\]
Breaking $\partial( (0,t) \times B_\nu(r) )$ into pieces, we deduce that
\[
\left| \int_{\{t\}\times B_\nu(r)} \omega \ - \ 
\int_{\{ 0 \}\times B_\nu(r)} \omega \right|  \ = \ 
\left| \int_{(0,t)\times \partial B_\nu(r)} \omega \, \right|
\]
where $\{s\}\times B_\nu(r)$ is understood to have the standard orientation for $s=0,t$ 
(rather than the orientation inherited as part of $\partial((0,t)\times \partial B_\nu(r))$.)
The left-hand side of this identity is just the left-hand side of \eqref{L3.c} in slightly different notation,
so it suffices to estimate the right-hand side,
which can be written more explicitly as
\[
\left| \int_{(0,t)\times \partial B_\nu(r)} \omega(U)(\tau_0, \tau_1)\  d\mathcal H^2 \right|,
\]
where $\tau_0(y),\tau_1(y)$ is a
properly oriented orthonormal basis for $T_y((0,t)\times \partial B_\nu(r))$,
and $\omega(U)(\tau_0, \tau_1)$ at a point $y$
denotes the number obtained by the two-form $\omega(U(y))$ acting on the
vectors $\tau_0(y), \tau_1(y)$.
Thus it suffices in fact to show that 
\be
|\omega(U)(\tau_0, \tau_1)| \le \max\{1,\lambda^{-1}\} e^{3d}_{\e,\lambda}(U)
\qquad\mbox{ for any orthonormal vectors $\tau_0,\tau_1$.}
\label{Bbis}\ee
To prove this, note that
$\omega(\tau_0, \tau_1)$
at a point $y$ is just\footnote{This is easily verified by fixing an orthonormal basis 
$\tau_0, \tau_1, \tau_2$ for $\R^3$, and then noting that
$\omega(\tau_0, \tau_1) = 
\partial_0 (j(U)+A)_1 - \partial_1(j(U)+A)_0$, where $\partial_i, j(U)_i$ and $A_i$ are all written with 
respect to the basis $\{\tau_i\}$.}
 the two-dimensional vorticity of the restriction of $U$
to the (suitably oriented) plane through $y$ spanned by $\tau_0,\tau_1$, 
and so \eqref{Bogcor} implies that 
$|\omega(U)(\tau_0, \tau_1)|$ is bounded by $\max\{1,\lambda^{-1}\} $
times 
the $2$-dimensional  energy $e^\nu_{\e,\lambda}$ of the restriction of  $U$ to the same plane, and this
clearly implies \eqref{Bbis}.
\end{proof}

Using the above lemma, we complete the

\noindent
\begin{proof}[Proof of proposition \ref{prop2}]
Let $g = -f'$, where $f$ is the function appearing in the definition of $\D^\nu_m$.
Then the choice  \eqref{f.def} of $f$ implies that
\be\label{deriv}
0 \le g(r)\leq Cr^{2},
\quad\quad
\int_{0}^{R}g(r) dr=1.
\ee
Then
\begin{align}
\int_{B_\nu(R)}\omega(\yn) f(|\yn|) d\yn
&=
\int_{B_\nu(R)}\omega(\yn)\left(\int^{R}_{\abs{\yn}}g(s)ds\right)d\yn \nonumber \\
&=\int_{B_\nu(R)}\int^{R}_{0} \omega(\yn)\chi_{\{\abs{\yn}<s\}}g(s)\, ds\  d\yn
\nonumber\\
&=\int^{R}_{0}g(s)\left(\int_{B_\nu(s)}\omega(\yn)d\yn\right) ds.
\label{step3a}\end{align}
%Also, in view of \eqref{deriv}, it is clear that $\pi m = \int_0^R g(r) \pi m \ dr$, 
%and it follows that
%\be
%\D^\nu_m(U;R) \ = \left| \int_0^R  g(r) \left( \int_{B_r} \omega \ dy  \ - \ \pi m  \right) \ dr  \right|
%\label{rewriteD}\ee

It follows from this and the definition \eqref{defect} of $D^\nu_m$ that 
\begin{align*}
\abs{D_{\nu}(U(t))-D_{\nu}(U(0))}&= \abs{\int_{B_\nu(R)} f(r)\big(\omega(U(t))-\omega(U(0))\big)}\\
&=
\abs{\int^{R}_{0} g(r)\left( \int_{B_\nu(r)} \omega(U(t))-\omega(U(0)) \right) dr }.
\end{align*}
Then by \eqref{deriv} and Lemma \ref{lemma3}, and recalling the definition 
of $e^{3d}_{\e,\lambda}(U)$,
\begin{align*}
\abs{D_{\nu}(U(t))-D_{\nu}(U(0))}
&\leq 
C\int^{R}_{0}r^{2}\left(\int_{(0,t)\times \partial B_\nu(r)}
e^{3d}_{\e,\lambda}(U)  \ d\mathcal H^2 \right)dr\\
%& = C \int_{(0,t)\times B_{\nu}(R)} |\yn|^2 e^{3d}_{\e,\lambda}(U)dy^{\nu}dy^0\\
&=
C \int_{(0,t)\times B_{\nu}(R)}
|\yn|^2 \left[ 
 \frac{\abs{D_{0}\phi}^{2}}2
+\frac{\e^{2}}2(F_{01}^{2} + F_{02}^2) +e^\nu_{\e,\lambda}(U)
\right]
dy^{\nu}dy^0,
\end{align*}
and this  immediately implies \eqref{changeD}.
\end{proof}

\subsection{Lower Energy Bounds}
A large part of the proof of Proposition \ref{prop1} is contained in  the following lemma.

\begin{lem}
Given a smooth configuration $U= (\phi, A)$ on $B_\nu(R)\subset \R^2$, for every $C_1>\frac 2R$ there exists
$C_2, \e_0>0$ such that
if $0<\e<\e_0$ and  
\be
\int_{\partial B_\nu(s)} e_{\e,\lambda}^\nu(U) \le C_1 \qquad\mbox{ for some }s \in (\e, R - \frac 1{C_1}),
\label{bdbd1}\ee
%for some $s\in (0,R)$, 
then $n(s) := \deg( \frac \phi{|\phi|}; \partial B_\nu(s))\in \Z$ is well-defined, and
\be
\int_{B_\nu(R)} e^\nu_{\e,\lambda}(U) \ge  \calE^\lambda_{n(s)} - C_2 \e ^2.
%\in \Z$ such that $\int_{B_r}\omega = \pi n(r) + O(\e)$}. 
\label{extendb}\ee
\label{ext}\end{lem}

We defer the proof of this until the end of this section. We note however that 
when $\lambda = 1$, a weaker version of the conclusion (with an error term of order $\e$ rather than $\e^2$)
follows immediately from Lemma \ref{lemma2} and
the fact that $\calE^1_n = \pi |n|$. 
\noindent

\begin{proof}[Proof of of proposition \ref{prop1}]
{\bf 1.}
We may assume that
\be\label{star2}
\int_{B_\nu(R)} e^\nu_{\e,\lambda}(U)\leq  \calE^\lambda_m
\ee
as otherwise the conclusion of the proposition is immediate. Also,
standard density arguments allow us to assume that $U$ is smooth.

Consider balls $B_\nu(s)=B_{s}, \e\leq s\leq R$.  We say that 
%\begin{itemize} \item 
$s$ is {\em good} if it satisfies the hypotheses of
Lemma \ref{ext} for some $C_1 > \frac 2R$ to be chosen below,
so that $\e <s < R - \frac 1{C_1}$ and 
\be
\int_{\partial B_{s}} e^\nu_{\e,\lambda}(U)d\mathcal H^{1}\leq C_1.
\label{good.ess}\ee
If $s$ is not good, then it is said to be {\em bad}.
%\end{itemize}
Lemma \ref{ext} %and \ref{lemma2} together imply 
implies that there exists some $C_2$ (depending on $C_1$) such that  if $s$ is good,  then
\[
\int_{B_s} e^\nu_{\e,\lambda}(U) \ge  \calE^\lambda_{n(s)} - C_2 \e^2
\qquad\quad
\mbox{for $n(s)= \deg( \frac \phi{|\phi|}; \partial B_s)$}. %\ = \pi^{-1}\int_{B_s}\omega + O(\e)$}. 
\]
Because  $n\mapsto \calE_n^\lambda$ is increasing for $n\ge 0$ by hypothesis, 
see \eqref{monotone}, the proposition follows if
%it therefore suffices to prove that 
\be
\mbox{there exists some good $s$ such that $n(s) \ge m$.}
\label{mustshow}\ee
We therefore assume that \eqref{mustshow}
does not hold, which in view of Lemma \ref{lemma2}  and the definition of good $s$ implies that
\be
\int_{B_{s}}\omega
\le \pi (m-1) + C \e\qquad \mbox{ for all good }s,
\label{toward}\ee
(for $C$ depending on $C_1$). We will show that this implies the desired lower bound \eqref{p1.conc}.
Toward this end, first note that, owing to \eqref{star2},
\begin{align*}
\calE_m^\lambda
\geq \int^{R}_{\e}(\int_{\partial B_{s}} e^\nu_{\e,\lambda}(U) d\mathcal H^{1}  )ds
&\geq \int_{{\mbox{\scriptsize bad}}}(\int_{\partial B_{s}}  e^\nu_{\e,\lambda}(U) 
d\mathcal H^{1}  )ds\\
%&\geq \int_{{\mbox{\scriptsize  bad}}} C_1 ds\\
&\ge C_1
 \abs{\{ s\in (\e, R) : \ \eqref{good.ess}\mbox{ fails } \}  }.% \abs{\{\mbox{bad}\}}.
\end{align*}
Hence
\[
 \abs{\{ s\in (\e, R) :  \ \eqref{good.ess}\mbox{ fails }\}}\leq \frac {\calE_m^\lambda}{C_1}.
\]
As a result, if $0<\e<\e_0$ and $\e_0\le \frac 1{C_1}$, then
\be\label{step1}
| \{ s \in (0, R) : \mbox{$s$ is bad}\} | \ \le  \e+\frac 1{C_1} +  \abs{\{ s\in (\e, R- \frac 1{C_1}) :  \ \eqref{good.ess}\mbox{ fails }\}} \le \frac {(2+\calE_m^\lambda)}{C_1}.
\ee

{\bf 2.}  
It follows from \eqref{step3a} and the definition \eqref{defect} of $\D^\nu_m$ that if $\D^\nu_m(U;R)\le \kappa_1$, then
\[
\pi m - \kappa_1 \le 
\int^{R}_{0}g(s) \left(\int_{B_{s}}\omega\, d\yn  \right) ds.
\]
Then we deduce from \eqref{toward} that 
\begin{align*}
\pi m - \kappa_1 
&\le 
 \int_{\mbox{\scriptsize good } s} g(s)\left(\int_{B_{s}}\omega \, d\yn\right)ds+ 
 \int_{\mbox{\scriptsize bad } s} g(s)\left(\int_{B_{s}}\omega\, d\yn \right)ds\\
% &\le
%\int_{\mbox{\scriptsize good } r} g(r)\left(2\pi(m-1( +  O(\e)\right)dr+ \int_{\mbox{\scriptsize bad } r} g(r)\left(\int_{B_{r}}\omega\, dx \right)dr \nonumber \\
%&\le
%2\pi(m-1)+O(\e) + \int_{\mbox{\scriptsize bad } r} g(r)\left(\int_{B_{r}}\omega\, dx \right)dr .
&\le \pi(m- 1) + C \e +
% \int_{\mbox{\scriptsize bad } s} g(s)\left(\int_{B_{s}}\omega\, dx \right)ds.
|\{ \mbox{ bad }s\}| \ \|g\|_\infty  \ \sup_{\mbox{\scriptsize bad}\ s} \left(\int_{B_{s}}\omega d\yn
\right).
%\label{step3}
\end{align*}
Rearranging and using \eqref{step1}, we find that
\be
 \ \sup_{\mbox{\scriptsize bad}\ s} \left(\int_{B_{s}}\omega\  d\yn
\right)
\ \ge \  
\frac{\pi  - \kappa_1 - C \e}  {  \|g\|_\infty \  |\{ \mbox{ bad }s\}|  } 
\ \ge \   \frac{ C_1(\pi  - \kappa_1 - C \e)}{  \|g\|_\infty \ (2+ \calE_m^\lambda ) } .
\label{almost}\ee
Also, it follows from \eqref{Bogcor} that 
\[
\int_{B_\nu(R)} e_{\e,\lambda}^\nu(U)  \ge  \min\{1,\lambda\} 
 \ \sup_{\mbox{\scriptsize bad}\ s} \left(\int_{B_{s}}\omega(U)\, d\yn\right).
\]
Then \eqref{p1.conc} 
follows from \eqref{almost} for all sufficiently small $\e$, if we choose
$\kappa_1 = \frac \pi 2$ and $C_1 \ge \frac R2$ such that 
$C_1 \ge \max\{1,  \lambda^{-1}\} \|g\|_\infty  \calE^\lambda_m(2+\calE^\lambda_m)
$ for example.
\end{proof}

We now prove the lemma that was used above to guarantee a nearly sharp lower 
energy bound for a ball bounded by a ``good radius".

%The idea is to show that if $\int_{\partial B_r}e^\nu_{\e,\lambda}(U)$ is small,  then we can find a configuration $\tilde U$ on all of $\R^2$  that agrees with $U$ on $B_r$, and whose energy outside ofwhose energy is not 

\begin{proof}[Proof of Lemma \ref{ext}]
We have assumed that $U$ satisfies 
\be
\int_{\partial B_\nu(s)} e^\nu_{\e,\lambda}(U) \ d\mathcal  H^1 \le C_1.
\label{bdbd2}\ee
It follows from this and Lemma \ref{geq.delta} that for $\e$ small enough,
\be
\big| \, 1-|u| \, \big| \ \le \ \frac 12\  \mbox{ on }\partial B_\nu(s)
\label{moduok}\ee
and hence that $n =n(s) = \deg(\frac \phi{|\phi|};\partial B_\nu(s))$
is well-defined. (In fact $\int_{B_\nu(s)}\omega = n + O(\e)$, by Lemma \ref{lemma2}).
 
%We must prove that
%\[
%\int_{B_\nu(s)} e^\nu_{\e,\lambda}(U) \ \ge \calE^\lambda_n - C \e^2.
%\]

{\bf 1.}  We first claim that there is a configuration $\tilde U$ on $\R^2$
such that $\tilde U = U\mbox{ in }B_\nu(s)$,
\be
\tilde U\in H_n, \quad\quad \quad\quad\mbox{ and }\quad
\int_{\R^2\setminus B_\nu(s)}e^\nu_{\e,\lambda}(\tilde U) < C  \e  \int_{\partial B_\nu(s)} e^\nu_{\e,\lambda}(U)  d\mathcal H^1.
\label{tU1}\ee

Although our definition of $H_n$ requires that $\tilde A\in H^1_{loc}$,
it suffices to construct $\tilde U$ such that $\tilde A\in L^1_{loc}$
and the distributional exterior derivative satisfies  $d\tilde A = F_{12} dy^1\wedge dy^2$,
with $F_{12}\in L^2$, since any such $\tilde U$ can be approximated arbitrarily well by
smooth (hence $H^1_{loc}$) functions, via a standard mollification procedure for example.

{\bf Definition of $\tilde U$}.
We will write $\tilde U$ on $\R^2\setminus  B_\nu(s)$ in polar coordinates $(r,\theta), r\ge s, \theta\in \R/2\pi\Z$.
First we write $U$ on $\partial B_\nu(s)$ in the form
\begin{align*}
\phi(s ,\theta )&=\rho (\theta) e^{i q(\theta)} \\
A(s, \theta ) &=A_{r}(\theta)dr+A_{\theta}(\theta)d\theta
\end{align*}
for certain smooth functions $\rho,A_r, A_\theta:\R/2\pi\Z\to \R$ and 
$q:\R/2\pi\Z \to \R/2\pi\Z$. Note that $\rho(\theta) \ge \frac 12$ for every $\theta$, and
that
\be
\int_0^{2\pi}q'(\theta)\, d\theta = 
2\pi \deg(\frac \phi{|\phi|};\partial B_\nu(s)) \ =  \  2\pi n.
\label{Udeg}\ee
We define $\tilde U = (\tilde \phi, \tilde A)$ as follows:
\begin{itemize}
\item $\tilde U = U$ in $B_\nu(s)$.
\item If $s< r < s +\e$, then
\begin{align*}
\tilde \phi(r,\theta)&= \left[ \rho + \frac{r-s}\e( 1-\rho)\right] e^{i q},\qquad 
\tilde A(r,\theta)=\left[A_{\theta} + \frac {r-s}\e ( q' - A_\theta)\right]d\theta.\end{align*}
\item  If $r\ge s+\e$, then $\tilde \phi(r,\theta) = e^{iq(\theta)},\ \ 
\tilde A(r,\theta) = q'(\theta) d\theta$.
\end{itemize}
%\begin{itemize}
%\item (1st annulus) If $R\leq r\leq R+\e$, then
%\begin{align}
%\tilde u(r,\theta)=\rho (\theta) e^{i u(\theta)},\quad \tilde A(r,\theta)=\psi(r)A_{r}(\theta)dr+A_{\theta}(\theta)d\theta,
%\end{align}
%where $\psi$ is a smooth cutoff with $\psi=1$ on $B_{R}$ and $\psi=0$ on $B^{c}_{R+\frac{\e}{2}}$.
%\item (2nd annulus) If $R+\e\leq r\leq R+2\e$, then
%\begin{align}
%\tilde u(r,\theta)&=h(r,\theta) e^{i u(\theta)},\quad \tilde A(r,\theta)=A_{\theta}(\theta)d\theta,\\
%h(r,\theta)&=\left(\frac {1}{\rho(\theta)}\left(\frac{r-R-\e}{\e}\right)+\frac{R+2\e-r}{\e}\right) \rho(\theta).
%\end{align}
%\item (3rd annulus) If $R+2\e\leq r\leq R+3\e$, then
%\begin{align}
%\tilde u(r,\theta)&=e^{i u(\theta)},\quad \tilde A(r,\theta)=\left(\frac{r-R-2\e}{\e}\partial_{\theta}u(\theta)+\frac{R+3\e-r}{\e}A_{\theta}(\theta)\right)d\theta.
%\end{align}
%\end{itemize}

It is standard, and easy to check, that %although  $\tilde A_r$ may be discontinuous
%across $\partial B_s = \{ (r,\theta)  : r = s\}$, 
the distributional exterior derivative
of $\tilde A$ satisfies $d\tilde A = d(\tilde A_r dr + \tilde A_\theta d\theta)
= (\partial_r \tilde A_\theta-\partial_\theta \tilde A_r) dr\wedge d\theta$ in $\R^2$,
despite the possible discontinuity of $\tilde A_r$ across $\{ (r,\theta) : r=s\}$.
In particular $d\tilde A  \in L^2_{loc}(\R^2)$.
Note also that  $e^\nu_{\e,\lambda}(\tilde U) = 0$ outside of $B_{s+\e}$,
so $\tilde U$ is a finite-energy configuration.
It then follows from \eqref{Udeg} that $\tilde U \in H_n$.

\medskip

{\textbf{Energy of $\tilde U$.}}  
Since as noted above $e^\nu_{\e,\lambda}(\tilde U)=0$ outside $B_{s+\e}$, to
complete the proof of  \eqref{tU1} it suffices to estimate the energy of $\tilde U$ in the
annulus $s<r<s+\e$. So we henceforth restrict our attention to this set.

Writing $\tilde \phi = \tilde \rho e^{i\tilde q}$ and noting from \eqref{moduok} that
$\frac 12 \le \tilde \rho \le \frac 32$,
we estimate
\begin{align*}
|D_{\tilde A}\tilde \phi|^2 
&=
\frac 1 {r^2}(\partial_\theta \tilde \rho)^2 + (\partial_r \tilde \rho)^2  + \tilde\rho^2\left[ \frac {(\partial_\theta \tilde q - \tilde A_\theta)^2}{r^2} + (\partial_r \tilde q - \tilde A_r)^2\right]\\
&\le
\frac 1{r^2}\rho'(\theta)^2 + \frac{(1-\rho(\theta))^2}{\e^2} +  C \frac { (q'(\theta)- A_\theta(\theta))^2}
{r^2}
\end{align*}
and similarly
\[
\e^2 |d\tilde A|^2 = \left| (q' - A_\theta) dr\wedge d\theta\right|^2 = \frac 1{r^2} (q' - A_\theta)^2
\]
and clearly $(\tilde \rho^2 -1)^2 \le (\rho(\theta)^2-1)^2$.
We combine these and conclude, after noting  that $(\rho -1)^2 \le (\rho^2-1)^2$ for $\rho\ge 0$ and again using \eqref{moduok}, 
that for $s\le r \le s+\e$,
\begin{align*}
e^\nu_{\e,\lambda}(\tilde U)(r,\theta)
\ & \le  \ 
C(\lambda) \left( \frac 1{2 s^2} \rho'(\theta)^2 +  \frac {\rho^2(\theta)(q' - A_\theta)^2}{s^2}
+ \frac \lambda{8\e^2}(1-\rho^2(\theta))^2\right)
\\
&\le
C(\lambda) \, e^\nu_{\e,\lambda}(U)(s,\theta).
\end{align*}
Thus
\begin{align*}
\int_0^{2\pi} \int_s^{s+\e} e^\nu_{\e,\lambda}(\tilde U)(r,\theta) \ r\, dr\, d\theta
&\le
C(\lambda)\int_s^{s+\e} \frac r s \ dr \int_0^{2\pi} e^\nu_{\e,\lambda}(U)(s,\theta) \, s\,  d\theta
\\
&\le C(\lambda) \ \e \  \int_{\partial B_\nu(s)} e^\nu_{\e,\lambda}(U) d\mathcal H^1
\end{align*}
which completes the proof of \eqref{tU1}.

{\bf 2.} Now we compute
using the definition of $\calE^\lambda_n$ and \eqref{bdbd2},
\begin{align}
\calE^\lambda_n 
\le
\int_{\R^2} e^\nu_{\e,\lambda}(\tilde U) 
&=
\int_{B_\nu(s)} e^\nu_{\e,\lambda}( U) \nonumber 
+
\int_{\R^2\setminus B_\nu(s)} e^\nu_{\e,\lambda}(\tilde U) 
\\
&=
\int_{B_\nu(R)} e^\nu_{\e,\lambda}( U)  
- \int_{B_\nu(R)\setminus B_\nu(s)}e^\nu_{\e,\lambda}( U)+ C \e.
\label{exta1}
\end{align}
As a result, the conclusion\eqref{extendb} follows unless
\be\label{improved}
C\e 
\ge 
\int_{B_\nu(R)\setminus B_\nu(s)} e^\nu_{\e,\lambda}( U)
\  = \ 
\int_s^R \int_{\partial B_\nu(\sigma)} e^\nu_{\e,\lambda}(U) d\mathcal H^1.
\ee
And if this holds, we can find some $\sigma\in (s, R)$
such that
\[\int_{\partial B_\nu(\sigma)} e^\nu_{\e,\lambda}( U)\ d\mathcal H^1 \ \le \  \frac {C\e}{R- s} \  \le \ C_1 C \e,
\qquad\quad\mbox{  since $s < R - \frac 1{C_1}$ by hypothesis.}
\]
Then by exactly the construction of Step 1, we can find some $\hat U$ that equals $U$
in $\partial B_\nu(\sigma)$, and such that
\[
\hat U\in H_{n(\sigma)}, \quad\quad \quad\quad
\int_{\R^2\setminus B_\nu(s)}e^\nu_{\e,\lambda}(\hat U) < C  \e  \int_{\partial B_\nu(\sigma)} e^\nu_{\e,\lambda}(U)  d\mathcal H^1\ \le C \e^2
\]
Note also that it follows from \eqref{improved}, the fact that $|\omega(U)|\le C e^\nu_{\e,\lambda}(U)$ and Lemma \ref{lemma2}
that $n(\sigma) = n(s) = n$. So by arguing exactly as in \eqref{exta1} we find that
\[
\calE^\lambda_n \ \le  \ \int_{\R^2}e^\nu_{\e,\lambda}(\hat U) \  
\ \le \  
C \e^2 + \int_{B_\nu(\sigma)}e^\nu_{\e,\lambda}(U)
\ \le \ 
C \e^2 + \int_{B_\nu(R)}e^\nu_{\e,\lambda}(U),
\]
completing the proof of the lemma.
\end{proof}

To close this section, we record for future reference the fact that 
Lemma \ref{ext} holds on domains more general than balls; this will
be used in the proof of Theorem \ref{T.exist}. Although we state the result for a
square, which is what we need,  it is clear that the proof remains
valid for any domain that is bi-Lipschitz homeomorphic to a ball,
with a constant depending on the domain. For simplicity, we prove the lemma
with error terms of order $\e$ rather than $\e^2$, as this suffices for 
our later application.

\begin{lem}
Given a configuration $U= (\phi, A)$ on an open set containing $Q_s = (-s,s)^2 \subset \R^2$, for every $C_1>0$ there exists
a $C_2$ such that
if 
\be
\int_{\partial Q_s} e_{\e,\lambda}(U) \le C_1
\label{bdbd1a}\ee
then $n(s):= \deg(\frac \phi{|\phi|};\partial Q_s)\in \Z$ is well-defined and 
\be
\int_{Q_s} e^\nu_{\e,\lambda}(U) \ge  \calE^\lambda_{n(s)} - C_2 \e.
%\qquad\mbox{for $n(s)\in \Z$ such that $\int_{B_s}\omega = \pi n(r) + O(\e)$}. 
\label{Qextendb}\ee
\label{Qext}\end{lem}

\begin{proof}
If $\Psi:B\to Q$ is a Lipschitz map between subsets of $\R^2$, and $U$
is a configuration on $Q$, we will write $\Psi^*U$ to denote the configuration
$(\phi \circ \Psi, \Psi^* A)$ on $B$. Note that 
\be
e^\nu_{\e, \lambda}(\Psi^* U)(\yn) \ \le \max\{ 1,  \| D\Psi\|_\infty^4\} \ e^\nu_{\e,\lambda}(U)(\Psi(\yn))
\label{tsfmnrg}\ee
for $\yn\in B$.

Define $\Psi(0):=0,$ and $\Psi(\yn) :=\frac {|\yn| \yn}{ \max\{ |y^1|, |y^2|\}}$, $|\yn|\neq 0$, where  $|\yn|$ is the standard Euclidean norm. % (and $\Psi(0):= 0$ ). 
%{\norm{y}_{\infty}}y$
%, where $\norm{y}_{\infty}=\max(\abs{y_{1}},\abs{y_{2}})$ and $|y|$ is the standard Euclidean norm. 
Then $\Psi(B_\nu(s)) = Q_s$, and $\| D \Psi\|_\infty \le 
1 \le \|D (\Psi^{-1})\|_\infty \le \sqrt 2$.
It follows  from \eqref{bdbd1a} and \eqref{tsfmnrg} and a change of variables that
\[
\int_{\partial B_\nu(s)} e^\nu_{\e,\lambda}(\Psi^* U) \le C
\int_{\partial Q_s} e^\nu_{\e,\lambda}(U) \le C_1'.
\]
Thus $\deg(\frac \phi{|\phi|}; \partial Q_s ) = \deg( \frac \phi{|\phi|}\circ \Psi, \partial B_\nu(s)) =: n(s) \in \Z$
exists, and the proof of Lemma \ref{ext} shows that there exists a configuration $\tilde U$
on $\R^2$  such that $\tilde U = \Psi^* U$ on $B_\nu(s)$, 
\[
\tilde U\in H_n, \quad \int_{\R^2\setminus B_\nu(s)}e^\nu_{\e,\lambda}(\tilde U) < C \e 
\int_{\partial B_\nu(s)} e^\nu_{\e,\lambda}(\Psi^* U) \le  C_2' \e.
\]
Now define $\hat U := (\Psi^{-1})^* \tilde U$. Then $\hat U = U$ in $Q_s$, 
$\hat U\in H_n$. Moreover, again using \eqref{tsfmnrg}, we have
\[
\int_{\R^2\setminus Q_s} e^\nu_{\e,\lambda}(\hat U) \le
C \int_{\R^2\setminus B_\nu(s)} e^\nu_{\e,\lambda}(\tilde U)
\le C_2' \e
\]
and then \eqref{Qextendb} follows exactly as in the proof of Lemma \ref{ext}, see \eqref{exta1}.
\end{proof}

\section{Minimizers of the 2d Euclidean energy}\label{minimizer}

The main result of this section gives a criterion for existence of solutions
of the minimization problem \eqref{2dmin}:

\begin{thm}
Assume that $\lambda$ and $N$ are such that
\be
\calE^\lambda_N <  \min\{ \calE^\lambda_{n_1} + \cdots + \calE^\lambda_{n_i} \ : \ n_1+\cdots + n_i = N, \mbox{ at least two $n_j$ are nonzero} \} .
\label{A.hyp}\ee
Then there exists $U\in H_N$ such that 
$\int_{\R^2} e^\nu_{\e,\lambda}( U)=\mathcal \calE^{\lambda}_N$.

In particular, there exists $U\in H_{\pm1}$ minimizing $\int_{\R^2} e^\nu_{\e,\lambda}(\cdot)$ if $\frac 15 <\lambda < 5$.
\label{T.exist}\end{thm}

It is proved in \cite[Theorem I.2]{Riviere02} that $\int_{\R^2} e^\nu_{\e,\lambda}(\cdot)$ attains its minimum in $H_n$ for $|n|=1$
and for all $\lambda$ sufficiently large. Our argument is close in spirit to
that of \cite{Riviere02},
and we omit some details that are either standard or can be found in \cite{Riviere02}.

\begin{proof}
We will show that there exists $\e_0>0$, to be specified below, such that for all $\e\in (0, \e_0)$, there exists $U\in H_N$ such that $\int_{\R^2}e^\nu_{\e,\lambda}(U) = \calE^\lambda_N$. In view of scale invariance, see  \eqref{scale_e}, this will
establish the theorem.

{\bf 1}. We first remark that it follows from \eqref{A1} and \eqref{Bogcor} that $\calE^\lambda_m \ge \min\{1,\lambda\} \pi |m|$, and hence that the minimum \eqref{A.hyp} is indeed attained, and in fact there exists $\delta^\lambda_N>0$ such that 
\be
\mbox{ if }n_1+\cdots + n_i = N \mbox{ and at least two $n_j$ are nonzero, then }
\sum_{j=1}^i \calE^\lambda_{n_j} > (1+\delta^\lambda_N)\calE^\lambda_N. 
\label{deltadef}\ee
We may also assume that $\delta^\lambda_N \le1$.

Assume that $\e<\e_0$, and 
let  $(U_{k})\subset H_{N}$ 
be a minimizing sequence, so that $\int_{\R^2} e^\nu_{\e,\lambda}( U_{k}) \rightarrow \mathcal \calE_{N}^{\lambda}$ as $k \rightarrow \infty$. 
We further assume, without loss of generality, 
that $U_{k}$ is smooth and that   $\int_{\R^2} e^\nu_{\e,\lambda}( U_{k})< (1 + \frac 12 \delta^\lambda_N)\calE^\lambda_N$ for every $k$.

Let $\L := (\R\times \Z) \cup (\Z\times \R)$, and for $y\in\R^2$, let $\tau_y U_k(x) :=
U_k(x - y)$. Then it follows from Fubini's Theorem
that
\[
\int_{y\in (0,1)^2} \int_{\L} e^\nu_{\e,\lambda}(\tau_yU_k) =
\int_{y\in (0,1)^2} \int_{\tau_y\L} e^\nu_{\e,\lambda}(U_k) =
2\int_{\R^2} e^\nu_{\e,\lambda}(U_k) \le 3 \calE_N^\lambda
\]
for every $k$.
Thus, replacing $U_k$ by a suitable translation $\tau_{y_k}U_k$ for every $k$, we
may arrange that
\be
\int_{\L} e^\nu_{\e,\lambda}(U_k)  \le 3\calE_N^\lambda\quad\mbox{ for every $k$.}
\label{Lbound}\ee

{\bf 2}. Now for every $p=(p_1,p_2)\in \Z^2$, let $Q^p := p + (0,1)^2 = \{ (x+p_1, y+p_2) : (x,y)\in (0,1)^2\}$.
It follows from \eqref{Lbound} that 
\[
\int_{\partial Q^p} e^\nu_{\e,\lambda}(U_k)  \le 3\calE_N^\lambda\quad\mbox{ for every $k\in \mathbb N$ and
$p\in \Z^2$}.
\]
We can thus apply  Lemma \ref{Qext} on every square  $Q^p$ to find that
if $\e$ is small enough, then
$n(p,k) := \deg(\frac {\phi_k}{|\phi_k|}; \partial Q^p)$ is well-defined, and
\[
\int_{Q^p} e^\nu_{\e,\lambda}(U_k) \ge \calE_{n(p,k)}^\lambda - C \e.
\]
Since $\calE^\lambda_m \ge \min\{1,\lambda\}|m| $ for all $m$, we can fix $\e_0 = \e_0(\lambda,N)$ so small that  $\calE^\lambda_m - C \e \ge \calE^\lambda_m(1-\frac 14 \delta^\lambda_N)$ for all nonzero $m$ and all $\e<\e_0$. Then
\[
\int_{\R^2} e^\nu_{\e,\lambda}(U^k) 
\ \ge \ 
\sum_{ p : n(p,k)\ne 0} \calE_{n(p,k)}^\lambda - C\e
\ \ge \ 
(1-\frac 14 \delta^\lambda_N)\sum_{ p : n(p,k)\ne 0} \calE_{n(p,k)}^\lambda.
\]
Also, as noted in Section 3, $\omega(U_k)$ is integrable, so by Lemma \ref{lemma2}, the
definition of  $H_N$, and the additivity of degree,
\[
N = 
 \lim_{t\to \infty} \int_{ \cup_{|p|<t} Q^p} \omega
=
 \lim_{t\to \infty} \deg\Big( \frac {\phi_k}{|\phi_k|} ; \partial( \cup_{|p|<t} Q^p) \Big)
=
 \ \sum_{p\in \mathbb Z^2} n(p,k).
\]
Thus the definition of $\delta^\lambda_N$ implies that if at least two $n(p,k)$
are nonzero, then
\[
\int_{\R^2} e^\nu_{\e,\lambda}(U_k) \ge 
(1 - \frac 14 \delta^\lambda_N)(1+\delta^\lambda_N)\calE^\lambda_N
\ = \ (1 + \frac 34 \delta^\lambda_N - \frac 14(\delta^\lambda_N)^2) \calE^\lambda_N
\ge ( 1+ \frac 12 \delta^\lambda_N) \calE^\lambda_N
\]
for all $k$,
since we assumed that $\delta^\lambda_N \le 1$. But this is not the case for any $k$, by our choice of the sequence $U_k$.
We conclude for every $k$, 
\be
\mbox{there exists $p_0=p_0(k)\in \mathbb Z^2$ such that $n(p_0, k) = N$, $n(p,k)= 0$ for $p\ne p_0$.}
\label{p0}\ee
Replacing (again) $U_k$ by a suitable translation, we may assume that $p_0 = (0,0)$ for every
$k$.

{\bf 3}. The remainder of the existence proof is now standard,
and very similar points are treated in detail in Rivi\`ere \cite{Riviere02}, so we summarize the arguments only briefly.
First,
 if we impose the Coulomb gauge condition $\nabla\cdot A_k = 0$ for every $k$, then the
uniform energy bounds imply that the sequence $U_k$ is weakly precompact in
$(\dot H^1\cap H^1_{loc})\times (\dot H^1\cap H^1_{loc})$. 
We can thus extract a subsequence converging weakly to a limit $U = (\phi, A)$,
and standard lower semicontinuity arguments imply that
\be\label{liminf}
\int_{\R^2} e^\nu_{\e,\lambda}(U) \leq \lim \inf_{k} E^\nu_{\e,\lambda}(U_{k})=\calE_{N}^{\lambda}.
\ee
Next, using \eqref{p0} and \eqref{Lbound}
we can check that $n(p) = \deg(\frac \phi{|\phi|}, \partial Q^p)$ is well-defined and
that $n(p_0)= N$ and that $n(p) = 0$ if $p\ne p_0=(0,0)$. It follows
that $U\in H_N$ and hence that $U$ is an energy-minimizer in $H_N$.

{\bf 4}. Finally, 
since 
\[
\min\{ 1,\lambda\} |n|\pi \le \calE^\lambda_n \le \max\{ 1,\lambda\}\, |n|\pi
\]
it is easy  check that \eqref{A.hyp} is satisfied for $N=1$ as long as $\frac 15 <\lambda < 5$.
For example, if $1<\lambda<5$, then
$\calE^\lambda_1 = \calE^\lambda_{-1} \le \lambda\pi < 5\pi$. 
 Now consider nonzero integers $n_1,\ldots, n_i$  such that $n_1+\ldots +n_i = 1$, with at least two nonzero $n_i$.
If $|n_j|=1$ for any $j$, it is clear that $\sum \calE^\lambda_{n_j}= \sum \calE^\lambda_{|n_j|} > \calE^\lambda_1$,
and if $|n_j|\ge 2$ for all $j$, then
\[
\sum \calE^\lambda_{n_j} \ge \pi \sum |n_j| \ge 5 \pi,
\]
since $\sum |n_j|$ is odd and must be greater than $3$. The case $\frac 1 5<\lambda<1$ is similar.
\end{proof}

\section{abelian Higgs model: energy estimates in normal coordinates}\label{ahm}

In this section we consider the abelian Higgs model in the coordinate system introduced
in Section \ref{cov}. 

In particular, recall the map $\psi$ defined in \eqref{psi.def}, built around a timelike 
minimal surface parametrized by an embedding $H:(-T,T)\times S^1\to \R^{1+3}$ as described in Section \ref{cov}.
Given $T_0<T$, we henceforth restrict the domain of $\psi$
to a set of the form $(-T_1,T_1)\times S^1\times B_\nu(\rho_0)$.
We do this in such a way that
\be
 T_0<T_1 < T, \quad\quad
\mbox{$\psi$ is injective  
on $(-T_{1}, T_{1}) \times S^{1} \times B_{\nu}(\rho_{0})$}
\label{T1aa}\ee
and
\be
\psi^0(-T_1, y^1, \yn) < - T_0 , \quad
\psi^0(T_1, y^1, \yn)  > T_0 , \quad
\mbox{ for all }|\yn|\le \rho_0, y^1\in S^1,
\label{T1}\ee
where $\psi^0$ denotes the $0$th component of $\psi$,  corresponding to the $t$ variable.
\newline\indent
Given $T_0 <T$, and having fixed $T_1 \in (T_0,T)$ and $\rho_0$ as above, 
we will write
\be
\calN := \psi \left( (-T_1,T_1)\times S^1\times B_\nu(\rho_0) \right) \cap
\left( (-T_0,T_0)\times \R^3 \right).
\label{calN.def}\ee
%Given any $T_0<T$, we may do this in such a way that  \ where $T_1$ and $\rho_0$ are chosen so that  $\psi$ is injective   on $(-T_{1}, T_{1}) \times S^{1} \times B_{\nu}(\rho_{0})$. % with a smooth inverse $\phi$.

We will write $(g_{\a\b})$ to denote the metric tensor written in the
normal coordinate system
\[
g_{\a\b} := \eta_{\gamma\delta}\  \partial_\a \psi^\gamma \  \partial_\b \psi^\delta,\qquad\quad
(\eta_{\a\b}) := \mbox{diag}(-1,1,1,1)
\]
and we also use the notation
\be
\label{gmetric}
(g^{\a\b}) :=(g_{\a\b})^{-1},\qquad
g := \det(g_{\a\b}).
\ee  
Thus, in the normal coordinate system the abelian Higgs model takes the form
\begin{align*}
-\frac 1 {\sqrt{-g}} D_{\a}(g^{\a\b}D_{\b}\phi \sqrt{-g})+\frac {\lambda} {2\e^{2}} (\abs{\phi}^{2}-1)\phi=0,
\\%\label{ahm1gg}\\
-\e^{2}\frac 1{\sqrt{-g}} \partial_{\a}(F^{\a\b}\sqrt{-g})-   g^{\b\a}\ip{i\phi, D_\a\phi}
=0.%\label{ahm2gg}
\end{align*}
Here $\alpha,\beta$ run from $0$ to $3$, and we raise and lower indices with  $(g^{\a\b})$ and
$(g_{\a\b})$ respectively.
We  find it useful to write the above system as
\begin{align}
-D_{\a}(g^{\a\b}D_{\b}\phi)-b\cdot D\phi+V'_{\e}(\phi)=0 \label{ahm1g},\\
-\e^{2}\left(\partial_{\a}F^{\a\b} + g^{\b\nu}b^\mu F_{\mu\nu}\right)
-g^{\beta\alpha} \ip{i\phi, D_\alpha \phi} 
%+\mathcal Im( \phi g^{\b\a}\overline{D_{\a}\phi})
=0\label{ahm2g},
\end{align}
where
\be
b^{\b}=\frac{\partial_{\a}\g}{\g}g^{\a\b},\qquad\qquad V_{\e}(\phi)=\frac {\lambda} {8\e^{2}} (\abs{\phi}^{2}-1)^2.
%&F^{\a\b}=g^{\a\mu}g^{\b\nu}F_{\mu\nu}.
\label{b.def}\ee

\subsection{properties of the metric}

We will need a number of properties of the metric $(g_{\a\b})$.
These are mostly well-known and can be found in the proof of  \cite[Prop.1]{Jerrard09} for example.
First, 
\be
(g_{\a\b})(\yt, \yn) = 
\left(
\begin{array}{ll}
(\gamma_{ab})(\yt) &0\\
0& \mbox{Id}
\end{array}\right) +
\left(
\begin{array}{ll}
O(|\yn|) &O(|\yn|)\\
O(|\yn|)& O(|\yn|^2)
\end{array}\right) 
\label{glowerab}\ee
(in block $2\times 2$ form), where $(\gamma_{ab})$ was introduced in
Section \ref{S:n-g} and satisfies
\eqref{conformal}, \eqref{gamma3}.
Hence
\be
(g^{\a\b})(\yt, \yn) = 
\left(
\begin{array}{ll}
(\gamma^{ab})(\yt) &0\\
0& \mbox{Id}
\end{array}\right) +
\left(
\begin{array}{ll}
O(|\yn|) &O(|\yn|)\\
O(|\yn|)& O(|\yn|^2)
\end{array}\right).
\label{gupperab}\ee
In  \cite[Prop.1]{Jerrard09} it is further shown that
\be
(\partial_0 g^{\a\b})(\yt, \yn) = 
\left(
\begin{array}{ll}
O(1) &O(|\yn|)\\
O(|\yn|)& O(|\yn|^2)
\end{array}\right).
\label{dtgupperab}\ee
and that the vector field $b$ defined in \eqref{b.def} satisfies
\be
|b^\nu| := \sqrt{(b^2)^2+(b^3)^2} \le C|\yn|, \qquad\qquad |b^\tau| :=\sqrt{(b^0)^2+(b^1)^2}\le C.
%b^\beta \xi_{\b}\xi_{0}&\leq C(\abs{\xi_{\tau}}^{2}+\abs{y^{\nu}}^{2}\abs{\xi_{\nu}}^{2}),
\label{g3}\ee
All the above estimates are uniform in $(-T_{1}, T_{1}) \times S^{1} \times B_{\nu}(\rho_{0})$.

We remark that the estimate $|b^\nu|\le C|\yn|$ 
follows from the condition that $\Gamma$ is a minimal surface, and 
it is the only place in our argument that we directly invoke this assumption.

\subsection{energy}

The natural energy for \eqref{ahm1g}-\eqref{ahm2g} is obtained from the
stress-energy tensor
\begin{align*}
T_{\alpha\beta}&=\frac 12 g_{\alpha\beta}\L -\frac {\partial\L} {\partial {g^{\a\b}}}\\
&=\frac 12 g_{\alpha\beta}\L -\frac 12 \ip{ D_{\a}\phi,D_{\b}\phi} -\frac {\e^{2}}{2} g^{\mu\nu} F_{\a\mu} F_{\b\nu}.
\end{align*}
%Raising an index yields
%\begin{align*}
%T^{\a}_{\b} \ 
%      \ = \ \frac 12 \delta^{\a}_{\b}\L -\frac 12 g^{\a\ga} \ip{ D_{\ga}\phi ,D_{\b}\phi } -\frac {\e^2}2F^{\a\nu}F_{\b\nu}.
%\end{align*}
We will state estimates in terms of the energy density $e_{\e,\lambda}(U) := 2T^0_0$.
Explicitly,  
\begin{align*}
T^{0}_{0}&=
\frac 12 \L -\frac 12 g^{0 \ga} \ip{ D_{\ga}\phi ,D_{0}\phi } -\frac {\e^2}2F^{0\nu}F_{0\nu}
\\
&=
\left[\frac 14 g^{\mu\nu}\ip{ D_{\mu}\phi ,D_{\nu}\phi }  -\frac 12 g^{0 \ga} \ip{ D_{\ga}\phi ,D_{0}\phi } \right]
+ \frac {\e^2}8  \left[ F_{\a\b}F^{\a\b} - 4 F^{0\b}F_{0\b}
\right]
+\frac{\lambda}{16\e^2}(|\phi|^2-1)^2.
%\\
%&=
%\frac 14 a^{\mu\nu}( D_{\mu}\phi ,D_{\nu}\phi ) 
%+ \frac {\e^2}8  \left( F_{ij}F^{ij} - 2 F^{0 j}F_{0 j }
%\right)
%+\frac{\lambda}{8\e^2}(|\phi|^2-1)^2,
\end{align*}
We define $(a^{\a\b})$ so that $a^{\a\b} \xi_\a\xi_\b  = g^{\a\b}\xi_a\xi_b - g^{0\b} \xi_0\xi_b$,
which leads to
\be
a^{00}=-g^{00}, \ \  \quad\quad a^{i0} = a^{0j} = 0, \quad\quad\ \ a^{ij}= g^{ij}, \quad i,j = 1,\ldots, 3.
\label{aab.def}\ee
With this notation,
\be\label{eeplambda}
e_{\e,\lambda}(U) \ : = \ 2T^{0}_{0}=
\frac 12 a^{\mu\nu} \ip{ D_{\mu}\phi ,D_{\nu}\phi } 
+ \frac {\e^2}4  \left( F_{\a\b}F^{\a\b} - 4 F^{0 \b}F_{0 \b }
\right)
+\frac{\lambda}{8\e^2}(|\phi|^2-1)^2.
\ee
We also remark that
\be
T^{j}_{0} = 
-\frac 12 g^{j\ga} \ip{ D_{\ga}\phi ,D_{0}\phi } -\frac {\e^2}2F^{j\nu}F_{ 0 \nu},
\label{Tj0}\ee
where we remind the reader that repeated latin indices are summed from $1$ to $3$.
We will need the following
\begin{lem}
There exist constants $0<c\le C$ such that
for every $U= (\phi, A)$, 
%$e_{\e,\lambda}(U)$ satisfies
\begin{align}
(1- C|\yn|^2)e_{\e,\lambda}^\nu(U) + c \Big[ |D_\tau \phi|^2+  \e^2|F_\tau|^2  \Big]
&\le \  e_{\e,\lambda}(U) 
\nonumber\\ 
&\le \  
(1+C|\yn|^2)e_{\e,\lambda}^\nu(U) + C \Big[ |D_\tau \phi|^2 + \e^2 |F_\tau|^2 \Big]
\label{pos2}\end{align}
uniformly for  $y = (\yt, \yn)\in (-T_1,T_1)\times S^1 \times B_\nu(\rho_0)$,
where
\be
\abs{F_\tau}^2
=
\sum_{0\leq \a\leq 3,\ \,  0\leq \b \leq 1}\abs{F_{\a\b}}^2,
\qquad \qquad \abs{F_\nu}^2= \abs{F_{2 3}}^2,
\label{Ftaunu}\ee
\be
|D_\tau \phi|^2 = |D_0\phi|^2 + |D_1\phi|^2, \qquad\qquad
|D_\nu\phi|^2 = |D_2\phi|^2 + |D_3\phi|^2, 
\label{Dtaunu}\ee
and 
\be
e_{\e,\lambda}^\nu(U) := 
\frac 12 |D_\nu\phi|^2 + \frac {\e^2}2|F_\nu|^2 +  V_\e(\phi)
\label{enuepla}\ee
\label{L:pos2}\end{lem}

\begin{proof}
The pointwise inequalities \eqref{pos2} follow from \eqref{gupperab}, \eqref{conformal}, \eqref{gamma3}
by routine computations.

\end{proof}

\subsection{A differential energy inequality}

We next prove

\begin{lem}\label{energy}
Assume that $U$ is a smooth solution of \eqref{ahm1g}, \eqref{ahm2g} on
$(-T_1,T_1)\times S^1\times B_\nu(\rho_0)$. Then
there exists $C>0$ such that 
\be
\partial_{0}e_{\e,\lambda}(U)\leq C\left(\abs{D_{\tau}\phi}^{2}+\e^2\abs{F_{\tau}}^{2}+\abs{y^{\nu}}^{2} \, e^\nu_{\e,\lambda}(U) %(\abs{D_{\nu}\phi}^{2}+\e^2\abs{F_{\nu}}^{2}) 
\right)- 2  \,\partial_i \,T^i_0(U)
\label{diffineq}\ee
pointwise.
%where $|F_\tau|^2$ and $|F_\nu|^2$ are defined in \eqref{Ftaunu}. % and
\end{lem}

As is well-known, the tensor  $T^\a_\b$ satisfies an {\em exact} conservation law
$\partial_\a( T^\a_\b\sqrt{-g}) = 0$ for $\alpha = 0,\ldots, 3$. However,  \eqref{diffineq} is more useful for our purposes. Surprisingly,  it does not seem to be easy to derive \eqref{diffineq}
directly from the exact conservation law.

\begin{proof}
We  take the inner product of  \eqref{ahm1g} 
with $D_0\phi$ to find
\[%\be
-\ip{D_{\a}(g^{\a\b}D_{\b}\phi), D_0\phi} -\ip{b\cdot D\phi, D_0\phi} +\ip{V'_{\e}(\phi), D_0\phi}=0.
\]%\label{nrg0}\ee
Note that $\ip{V'_{\e}(\phi), D_0\phi} = \partial_0 V_\e(\phi)$.
Also, using the commutation relation
\[ %\be\label{commute}
[D_{\a},D_{\b}]=iF_{\b\a}.
\]
we compute
\begin{align*}
-\ip{D_{\a}(g^{\a\b}D_{\b}\phi), D_0\phi}
&=
-\partial_\alpha \ip{g^{\alpha\beta}D_\beta\phi, D_0\phi}
+ 
\ip{g^{\a\b}D_{\b}\phi , D_{\a} D_0\phi}
\nonumber \\
&=
-\partial_\alpha \ip{g^{\alpha\beta}D_\beta\phi, D_0\phi}
+ 
\ip{g^{\a\b}D_{\b}\phi , D_{0} D_\a\phi}
+
\ip{g^{\a\b}D_{\b}\phi , [D_\alpha, D_0]\phi}
\nonumber \\
&=
-\partial_\alpha \ip{g^{\alpha\beta}D_\beta\phi, D_0\phi}
+ 
\frac 12 \partial_0\ip{g^{\a\b}D_{\b}\phi , D_\a\phi}\\
&\quad
-
\frac 12 (\partial_0g^{\a\b})\ip{D_{\b}\phi , D_\a\phi}
+
F_{0\alpha} \ip{g^{\a\b}D_{\b}\phi ,i \phi}.
\label{nrg1}\end{align*}
Since $-g^{0\b}\xi_\b \xi_0 + \frac 12 g^{\a\b}\xi_a\xi_\b =  \frac 12 a^{\a\b}\xi_\a \xi_\b$,
by collecting terms we conclude that
\be
\begin{split}
\partial_0\left(\frac 12 a^{\a\b} \ip{ D_\a \phi, D_\b\phi}
+ V_\e(\phi)
\right) \, &= \, \ip{b\cdot D\phi, D_0\phi }
+ \partial_i \ip{g^{i\b}D_\b\phi, D_0\phi } \\
&\quad
+
\frac 12 (\partial_0g^{\a\b})  \ip{D_{\b}\phi , D_\a\phi}
-F_{0\a}g^{\a\b} \ip{D_\b\phi, i\phi}.
\label{nrg1a}
\end{split}\ee
We now rewrite the last term on the right-hand side. First, using the equation \eqref{ahm2g},
\begin{align}
F_{0\alpha}g^{\a\b} \ip{D_{\b}\phi ,i \phi}
&=
- \e^2
F_{0\a}\left(\partial_{\b} F^{\b\a} + g^{\a\gamma'}b^{\gamma} F_{\gamma\gamma'}\right).
\label{nrg2}
\end{align}
We will leave the second term as it is. As for the first term, note that
\begin{align}
F_{0\a}\partial_{\b} F^{\b\a} 
&=
\partial_\b(F_{0\a}F^{\b\a} ) 
- 
\partial_\b F_{0\a} F^{\b\a}\nonumber
\\
&=
\partial_\b(F_{0\a}F^{\b\a} ) 
+ (  \partial_\a F_{\b 0} + \partial_0 F_{\a\b} ) F^{\b\a}\nonumber
\\
&=
\partial_\b(F_{0\a}F^{\b\a} ) 
+ \partial_\a (F_{\b 0} F^{\b\a}) - F_{\b 0} \partial_\a F^{\b\a} 
- 
 \partial_0 F_{\a\b} F^{\a\b} \nonumber
 \\
 &=
2\partial_\b(F_{0\a}F^{\b\a} ) 
- F_{\b 0} \partial_\a F^{\b\a} 
- 
 \partial_0 F_{\a\b} F^{\a\b}
\nonumber \\
&=2 \partial_\b(F_{0\a}F^{\b\a} ) 
- 
F_{\b 0} \partial_\a F^{\b\a} 
-
\frac 12 \partial_0 (F_{\a\b} F^{\a\b})
+ \frac 12 
 \partial_0( g^{\a\gamma} g^{\b\gamma'}) F_{\a\b} F_{\gamma\gamma'}\nonumber.
\end{align}
Hence
\be\label{nrg3}
F_{0\a}\partial_{\b} F^{\b\a} = \partial_\b(F_{0\a}F^{\b\a} ) 
 -
\frac 14 \partial_0 (F_{\a\b} F^{\a\b})
+ \frac 14 
 \partial_0( g^{\a\gamma} g^{\b\gamma'}) F_{\a\b} F_{\gamma\gamma'}.
\ee
%Also,
%\begin{align*}
% \partial_0 F_{\a\b} F^{\a\b}
%= \partial_0
%F_{\a\b}  g^{\a\mu} g^{\b\nu} F_{\mu\nu}
%&=
%\frac 12  \partial_0 \left( F_{\a\b}  g^{\a\mu} g^{\b\nu} F_{\mu\nu} \right) - \frac 12 \partial_0( g^{\a\mu} g^{\b\nu}) F_{\a\b} F_{\mu\nu}\\
%&=
%\frac 12  \partial_0\left(F_{\a\b}   F^{\a\b}\right) - \frac 12 \partial_0( g^{\a\mu} g^{\b\nu}) F_{\a\b} F_{\mu\nu}.
%\end{align*}
We combine  \eqref{nrg1a}-\eqref{nrg3} and collect all terms involving 
$\partial_0$ on the left -hand side, to find that
\begin{align}
&\partial_0\left(\frac 12 a^{\a\b} \ip{ D_\a \phi, D_\b\phi}
+\frac{\e^2}{4}( F_{\a\b}F^{\a\b} - 4F_{0\a}F^{0\a})
+ V_\e(\phi)
\right) 
\nonumber \\
&\hspace{4em} 
= \,  
\partial_i\left(  g^{i\b} \ip{D_\b\phi, D_0\phi }  + \e^2 F_{0 \a} F^{i \a}\right)
+\ip{ b\cdot D\phi, D_0\phi }
+\e^2 F_{0\a} g^{\a\gamma'}b^{\gamma}F_{\gamma\gamma'} \label{nrg4}
\\
&\hspace{8em}
+\frac 12 (\partial_0g^{\a\b}) \ip{ D_{\b}\phi , D_\a\phi }
+ \frac {\e^2}4 
 \partial_0( g^{\a\gamma} g^{\b\gamma'}) F_{\a\b} F_{\gamma\gamma'}.
 \nonumber
\end{align}
Note that the left-hand side is just $\partial_0 e_{\e,\lambda}(U)$, and the first term 
on the right-hand side is $-2 \, \partial_i \ T^i_0(\phi, A)$. So we just need to estimate the
other terms on the right-hand side. First,
by \eqref{dtgupperab} and \eqref{g3}, which we recall is essentially the assumption that $\Gamma$ is minimal,
\[
\ip{ b\cdot D\phi, D_{0}\phi }+\frac 12 (\partial_{0}g^{\a\b}) \ip{D_{\a}\phi,D_{\b}\phi}
\leq C\left(\abs{D_{\tau}\phi}^{2}+\abs{y^{\nu}}^{2}\abs{D_{\nu}\phi}^{2} \right).
\]
%The above estimate is the place where we use the assumption that $\Gamma$ is minimal. 
Next,
\begin{align*}
 \frac {\e^2}4 \partial_{0}(g^{\a\gamma}g^{\b\gamma'})F_{\a\b}F_{\gamma\gamma'}&=  \frac {\e^2}2 (\partial_{0}g^{\a\gamma})g^{\b\gamma'}F_{\a\b}F_{\gamma\gamma'}\\
&\leq \frac {\e^2}2 \norm{g^{\b\gamma'}}_{L^{\infty}}\abs{ (\partial_{0} g^{\a\gamma}) F_{\a \b}F_{\gamma \gamma'}}\\
&\leq C{\e^2}\left(\abs{F_{\tau}}^{2}+\abs{y^{\nu}}^{2}\abs{F_{\nu}}^{2} \right),
\end{align*}
using \eqref{gupperab}-\eqref{dtgupperab}.
Finally, from \eqref{gupperab} and \eqref{g3} we similarly estimate
\begin{align*}
\e^2 F_{0\a} g^{\a\gamma'}b^{\gamma}F_{\gamma\gamma'}
%&={\e^2}F_{0j}g^{j\gamma'}\frac{\partial_{\a}\g}{\g }g^{\a\gamma}F_{\gamma \gamma'}\\
%&\leq{\e^2} \norm{g^{j\gamma'}}_{L^{\infty}}\frac{\partial_{\a}\g}{\g }g^{\a\gamma}F_{\gamma \gamma'}F_{0j}\\
&\leq C{\e^2}\left(\abs{F_{\tau}}^{2}+\abs{y^{\nu}}^{2}\abs{F_{\nu}}^{2} \right).
\end{align*}
\end{proof}

\subsection{weighted energy estimate}

In this subsection we establish an estimate that plays a key role
in the proof of Theorem \ref{mainthm}.
We first introduce some notation.

Given a configuration $U$ on $S^1\times B_\nu(R)$ for some $R>0$,
if $U$ is a configuration on $S^1\times B_\nu(R)$ for some $R>0$,
then we will  use the notation
\be
\D_m(U;R) := \int_{y^1\in S^1} | \D^\nu_m(U(y^1) ; R) | \ dy^1
\label{D.def}\ee
for $m\in \Z$, where here $U(y^1)$ denotes the function 
$\yn\mapsto U(y^1, \yn)$. We recall that  $\D^\nu_m$ is defined in \eqref{defect}.
%We will need to use results about $D^\nu_m$ from Section \ref{S:Dnu}. For this reason, throughout this section we restrict our attention to  pairs $\lambda>0, m\in \Z$ such that the hypothesis  \eqref{monotone} of Proposition \ref{prop1} is satisfied.

The main result of this section is 

\begin{prop}\label{prop3}
Assume that
$U$ is a smooth solution of \eqref{ahm1g}-\eqref{ahm2g}
on $(-T_1,T_1)\times S^1\times B_\nu(\rho_0)$ and
that $m$ is a nonzero integer such that \eqref{monotone} is satisfied.
Then there exist positive constants $c_*,C, \kappa_2$ and $\rho_1\in (0,\rho_0/2]$, independent
of $U$ and of $\e\in (0,1]$, such that the following
hold:
if we define
\begin{align}
\zeta_1(s) &:= \left.\int_{S^1} \left( \int_{B_\nu(\rho_1-c_*s)} (1+\kappa_2|\yn|^2)e_{\e,\lambda}(U) d\yn  - \calE^m_\lambda \right) dy^1\right|_{y^0=s} % &\  \le C\zeta_0
\label{pc1}\\
\zeta_2(s) &:= D_m(U(s ,\cdot); \frac 12 \rho_1) %&\ \le C\zeta_0
\label{pc2} %\qquad\qquad\mbox{ for $D_m$ defined in \eqref{D.def} above}
\\
\zeta_{3} (s)&:=
\left.\int_{S^1}\int_{B_\nu(\rho_1-c_*s)} \Big(\abs{D_{\tau} v}^{2}+\e^{2}\abs{F_{\tau}}^{2}+\abs{y^{\nu}}^{2}e_{\e,\lambda}^{\nu}(U) \Big)
 d\yn \ dy^1\ \ \right|_{y^0 = s}  %&\ \le C\zeta_0
\label{pc3}
\end{align}
(where the notation $|D_\tau\phi|, |F_\tau|$ etc is introduced in \eqref{Ftaunu}-\eqref{enuepla})
then 
\be
\zeta_i(s) \le C\max\{ \zeta_1(0),  \zeta_2(0), \e^2 \},\ \mbox{ for }\ i = 1,2,3, \quad 0<s <s_{max} := \min
\{ T_1, \rho_1 /2c_* \}.
\label{pa}\ee
\end{prop}

The proof follows that of Proposition 10 in \cite{Jerrard09}.

\begin{proof}
We fix $\rho_1 \in (0, \rho_0/2]$  and $\kappa_2>0$, independent of $\e, U$, such that 
there exists $c>0$ satisfying
\be
e_{\e,\lambda}^\nu(U) + |D_\tau \phi|^2+  \e^2|F_\tau|^2 
\le \  c^{-1}  e_{\e,\lambda}(U) 
\label{rho1}\ee
and
\be
(1+\kappa_2 |\yn|^2) e_{\e,\lambda}(U) \ge c(|D_\tau \phi|^2+  \e^2|F_\tau|^2) + (1+ |\yn|^2) e^\nu_{\e,\lambda}(U)
\label{kappa2}\ee
in $(-T_1,T_1)\times S^1 \times B_\nu(2\rho_1)$. This may be done due
to Lemma \ref{L:pos2}.
It is convenient to use the notation $\zeta_0 := \max\{ \zeta_1(0), \zeta_2(0)\}$ and 
\[
W_{\nu}(s) : =B_{\nu}(\rho_{1}-c_{\ast}s),
\qquad\qquad
W(s) :=S^{1}\times W_{\nu}(s)
\]
for $c_\ast$ to be fixed below.
\newline\noindent
{\textbf{1.}}  We show
\be \label{p3step1}
\zeta_{1}(s)\leq \zeta_{0}+C\int_{0}^{s} \zeta_{3}(s')\, ds'\quad\mbox{for $s\in (0,s_{\max}]$.}
\ee
Since  $\zeta_{1}(0)\leq \zeta_{0}$, it suffices to prove that $\zeta_{1}'(s)\leq C \zeta_{3}(s)$. To that end we compute
\begin{align*}
\zeta_{1}'(s)&=\int_{{s}\times W(s)} (1+\kappa_2\abs{y^{\nu}}^{2})\, \partial_{0}e_{\e}(U) \\
&\quad -c_{\ast}\int_{\{s\}\times S^{1}\times \partial W_{\nu}(s)} (1+\kappa_2\abs{y^{\nu}}^{2})e_{\e,\lambda}(U) =I+II.
\end{align*}
By Lemma \ref{energy}, % and integration by parts,
\begin{align*}
I&\leq 
C\int_{\{s\}\times W(s)} (1+\kappa_2 \abs{y^{\nu}}^{2})\left(\abs{D_{\tau}\phi}^{2}+\e^{2}\abs{F_{\tau}}^{2}+\abs{y^{\nu}}^{2}e^\nu_{\e,\lambda}(U)   \right)\\
&\qquad\qquad
- 2\int_{\{s\}\times W(s)} (1+\kappa_2 |\yn|^2)\partial_i\, T^i_0(U).
\end{align*}
We integrate by parts to find that
\begin{align*}
 2\left|\int_{\{s\}\times W(s)} (1+\kappa_2\, |\yn|^2)\partial_i\, T^i_0(U) \right|
&\leq 
4\kappa_2\int_{\{s\}\times W(s)} \abs{y^{\nu}}\abs{T^\nu_0(U)}\\
&\qquad
+ 2\int_{\{s\}\times S^{1}\times \partial W_{\nu}(s)} (1+\kappa_2 \abs{y^{\nu}}^{2})\abs{T^\nu_0} 
\end{align*}
where $T^\nu_0 : = ( T^2_0, T^3_0) = (T^{\nu1}_0, T^{\nu2}_0)$.
We see from the definition \eqref{Tj0} of $T^j_0$ and the uniform bounds \eqref{gupperab} on $(g^{\a\b})$ that 
\be
|T^\nu_0|  \ \le \ 
C\left( |D_\nu \phi| \ |D_\tau \phi| +  \e^2 |F_\nu| \ |F_\tau| \ \right)
\ \le \   
C\left( e^\nu_{\e,\lambda}(U) + |D_\tau\phi|^2 + \e^2|F_\tau|^2\right)
\label{Tnu0}\ee
and it follows from this and \eqref{rho1} that we may choose $c_*$ large enough that
\[
2\int_{\{s\}\times S^{1}\times \partial W_{\nu}(s)} (1+\kappa_2\abs{y^{\nu}}^{2})\abs{T^\nu_0} 
+ II \le 0. 
\]
Also, arguing as in \eqref{Tnu0}, we deduce that
\[
\abs{y^{\nu}}\abs{T_0^{\nu}}
\ \le \ C \left( |D_\tau\phi|^2 + \e^2 |F_\tau|^2 +  |\yn|^2 e^\nu_{\e,\lambda}(U) \right),
\]
and by combining these estimates, we find that $\zeta_1'\le C \zeta_3$, completing the proof of \eqref{p3step1}.
\newline\noindent
{\textbf{2.}}  
Next, recalling the definition \eqref{pc2}, \eqref{D.def} of $\zeta_2$,
\begin{align*}
\zeta_2(s) 
&\le 
\zeta_0 + \zeta_2(s) - \zeta_2(0) \\
&\le \zeta_0 + \int_{S^1} \left|\D^\nu_m(U(s, y^1); \frac 12 \rho_1) - \D^\nu_m(U(0, y^1);\frac 12 \rho_1) \right| d y^1,
\end{align*}
so it follows immediately from Proposition \ref{prop2}  that
\be \label{p3step2}
\zeta_{2}(s)\leq C\zeta_{0}+C\int_{0}^{s} \zeta_{3}(s')ds'\quad\mbox{for $s\in (0,s_{\max}]$.}
\ee
\newline\noindent
{\textbf{3.}}  
We next show that
\be \label{p3step3}
\zeta_{3}(s)\leq C(\zeta_{1}(s)+ \zeta_{2}(s)+ O(\e^2)) \quad\mbox{for $s\in (0,s_{\max}]$.}
\ee
Since \eqref{kappa2} implies that
\[
\zeta_1(s) \ge c \,\zeta_3(s) + \left.\int_{S^1}\left(\int_{W_\nu(s)} e_{\e,\lambda}^\nu(U) - \calE^\lambda_m\right)\right|_{y^0=s} 
\ = \ 
c \,\zeta_3(s) - |S^1| \calE^\lambda_m  \ +  \ \int_{\{s\}\times W(s)} e_{\e,\lambda}^\nu(U) ,
\]
it suffices to show that 
\be\label{p3step3a}
|S^1| \calE^\lambda_m-\int_{\{s\}\times W(s)}e_{\e,\nu}^\nu(U) \ \leq \  C\zeta_{2}(s)+O(\e^2).
\ee
The $s$ variable plays no role in this argument, so we regard it as fixed and do not display it.
We 
will say that
 $y^1\in S^{1}$ is {\em good} if
\[
|\mathcal D_m^{\nu}(U(y^1))|\ \leq \kappa_{1},\qquad \mbox{ where $\kappa_1$ was fixed  in Proposition \ref{prop1}},
\]
and $y^1$ is 
{\em bad} otherwise.
As usual we estimate the size of the bad set by Chebyshev's inequality:
\[
\abs{\{ y^1\in S^{1}: y^1 \quad\mbox{is bad}  \}}\leq \frac {1}{\kappa_{1}}\int_{\{s\}\times S^{1}} | \mathcal D_m^{\nu}(U)|\  dy^1 = C\zeta_{2}(s).
\]
So $\abs{\{ y^1\in S^{1}: y^1 \quad\mbox{is good}  \}}\geq |S^1|- C\zeta_{2}(s)$, and we
obtain the estimate we seek
by applying the lower energy bounds from Proposition \ref{prop1}
in the normal variables for every good $y^1$. Indeed,
\begin{align*}
\int_{\{s\}\times W(s)}e_{\e,\nu}(U) \  &\geq \int_{\{(s,y^1)\in S^{1}\ \mbox{is good}\}}\left (\int_{W_{\nu}(s)}e_{\e,\nu}(U)d\yn\right)dy^1\\
&\geq  (|S^1| - C\zeta_{2}(s))\ (\calE^\lambda_m - C\e^2)\\
&\geq |S^1|\calE^\lambda_m  -C\zeta_{2}(s)- C\e^2.
\end{align*}
Rearranging gives \eqref{p3step3a}.
\newline\noindent
{\textbf{4.}}  We gather the previous steps to conclude
\[
\zeta_{3}(s)\leq C(\zeta_{1}(s)+ \zeta_{2}(s)+ O(\e^2)) \leq C \zeta_{0}+ C\int_{0}^{s} \zeta_{3}(s')ds' + CO(\e)\leq  C \zeta_{0}+ C\int_{0}^{s} \zeta_{3}(s')ds', 
\] 
since $\zeta_{0}\geq \e^2$.  Then by Gronwall's inequality,
\[
\zeta_{3}(s)\leq C\zeta_{0}  \quad\mbox{for $s\in (0,s_{\max}]$},
\]
and hence from Steps 1 and 2,
\[
\zeta_{1}(s),\zeta_{2}(s)\leq C\zeta_{0}
\]
as needed.
\end{proof}

\section{proof of Theorem \ref{mainthm}}\label{proof_mainthm}

In this section we complete the proof of Theorem \ref{mainthm}. This involves,
among other things, combining
the weighted energy estimates of Proposition
\ref{prop3}, expressed in the normal coordinate system  and
effective near $\Gamma$, with energy estimates in the standard
coordinate system, effective away from $\Gamma$.

In this section we will write 
\begin{align*}
e_{\e,\lambda}(\calU,\eta) 
&\ = \  \frac 12 |D\vp|^2 + \frac {\e^2}2 |\F|^2 + \frac \lambda{8\e^2}(|\vp|^2-1)^2 ;\\
e_{\e,\lambda}(U,g) 
&\ = \  
\frac 12 a^{\mu\nu} \ip{ D_{\mu}\phi ,D_{\nu}\phi } 
+ \frac {\e^2}4  \left( F_{\a\b}F^{\a\b} - 2 F^{0 \b}F_{0 \b }
\right)
+\frac{\lambda}{8\e^2}(|\phi|^2-1)^2
\end{align*}
for the natural energy densities with respect to 
standard coordinates and normal coordinates respectively, where
we raise indices with  $(g^{\a\b})$ in the second expression.
It is straightforward to check from the definitions and \eqref{rho1} that
there exists a $C$ independent of $\calU$ and $\e\in (0,1]$ such that
\be
\frac 1C e_\e(\calU,\eta)(\psi(y)) \le \ e_\e(U,g)(y)  \ \le \ C e_\e(\calU,\eta)(\psi(y))
\label{comparable}\ee
for $y\in (-T_1,T_1)\times S^1\times B_\nu(\rho_1)$.

\begin{prop}\label{mainprop2}
Assume that $\Gamma, T_0, T_1, \rho_0, \calN$ are as in the statement of
Theorem \ref{mainthm}, and assume that $\lambda>0$ and $m\in \Z$
satisfy the conditions in Theorem \ref{mainthm}.
Let $\rho_1, \kappa_2$ denote the constants found in Proposition \ref{prop3}.

Let $\calU = (\vp, \A)$ solve the abelian Higgs model \eqref{ahm1}-\eqref{ahm2}
%in standard coordinates 
on $\R^{1+3}$, and let $U = \psi^*\calU$,
so that $U$ solves \eqref{ahm1g}-\eqref{ahm2g} on $(-T_1,T_1)\times S^1\times B_\nu(\rho_0)$.

Define
\begin{align*}
\tilde \zeta_1(s) &:= \left.\int_{S^1} \left( \int_{B_\nu(\rho_1/2)} (1+\kappa_2|\yn|^2)e_{\e,\lambda}(U,g) d\yn  - \calE^m_\lambda \right) dy^1\right|_{y^0=s} % &\  \le C\zeta_0
\\ %\label{pc1}\\
\tilde \zeta_2(s) &:= D_m(U(s ,\cdot); \frac 12 \rho_1) %&\ \le C\zeta_0
 % \label{pc2} %\qquad\qquad\mbox{ for $D_m$ defined in \eqref{D.def} above}
\\
\tilde \zeta_{3} (s)&:=
\left.\int_{S^1}\int_{B_\nu(\rho_1/2)} \Big(\abs{D_{\tau} \phi}^{2}+\e^{2}\abs{F_{\tau}}^{2}+\abs{y^{\nu}}^{2}e_{\e,\lambda}^{\nu}(U) \Big)
 d\yn \ dy^1\ \ \right|_{y^0 = s}  %&\ \le C\zeta_0
%\label{pc3}
\end{align*}
and
\[
\tilde \zeta_4(t) := 
\int_{(\{t\}\times \R^3) \setminus \calN_1}  e_{\e,\lambda}(\calU, \eta) dx
\]
where
\[
\calN_1 := \psi \left( (-T_1,T_1)\times S^1\times B_\nu(\rho_1/2) \right) \cap
\left( (-T_0,T_0)\times \R^3 \right).
\]
Finally, let $\zeta_0 := \max\{\tilde  \zeta_1(0), \tilde \zeta_2(0),\tilde \zeta_4(0), \e^2 \}$.

Then there exists a constant $C$, independent of $\calU$ and $\e\in (0,1]$, such that
\be
\tilde \zeta_i(s) \le C \zeta_0 
\mbox{ for $=1,2,3$ and $-T_1\le s\le T_1$},
\quad\quad
\tilde \zeta_4(t) \le C \zeta_0 \mbox{ for }-T_0\le t \le T_0.
\label{tcb}\ee
\end{prop}

We follow the proof of
\cite[Theorem 22]{Jerrard09}. We start
by presenting all the details, to illustrate
the basic argument,
and we refer to \cite{Jerrard09} for the final
part of the proof.

The proof will use some
standard energy estimates
for $\calU$ and $U$ respectively, which we recall for the reader's convenience.

\begin{lem}
Let $\calU$ be a smooth finite-energy solution of the abelian Higgs model \eqref{ahm1}-\eqref{ahm2} in standard coordinates on
$\R^{1+3}$. For any $a<b$ and any bounded Lipschitz function $\chi:(a,b)\times \R^{3}\to \R$
\be
\left|
 \int_{\{b\}\times \R^3} e_\e(\calU,\eta) \chi  \ dx -
 \int_{\{a\}\times \R^3} e_\e(\calU,\eta) \chi  \ dx
\right| \ \le \int_{(a,b) \times \R^3} e_\e(\calU,\eta) |D\chi| \ dx \ dt.
\label{ee1}\ee
%Also,  for any pair $a,b$ of real numbers and any open $A\subset\R^3$
%\be
%\int_{\{b\}\times A_{|b-a|}} e_\e(\calU,\eta) \ dx \ \le \  \int_{\{a\}\times A} e_\e(\calU,\eta) \ dx,
%\label{ee2}\ee
%where 
%$
%A_s:= \{x\in A : \mbox{dist}(x,\partial A)>s\}$.
\label{standard.ee1}\end{lem}

\begin{proof} 
Recall the energy identity
\be
\partial_t \,e_\e(\calU,\eta) - \partial_i \left( \delta^{ij} \ip{D_j\vp, D_0\vp } + \e^2 \F^{i\nu}\F_{0\nu}\right) = 0 
\label{ahm.ee}\ee
for solutions of \eqref{ahm1}-\eqref{ahm2}. This is standard and also can  be deduced from
\eqref{nrg4} (replacing $(g_{\a\b})$ by $(\eta_{\a\b})$). We integrate by parts and use the fact that 
$
\left|  \ip{D_i\vp, D_0\vp } + \e^2 \F^{i\nu}\F_{0\nu}\right|  \le e_{\e,\lambda}(\calU, \eta)
$ and routine estimates to deduce \eqref{ee1}. If $\chi$ has unbounded support, then  one can approximate it by functions
with compact support and use the fact that $\calU$ has finite energy to
pass to limits and obtain \eqref{ee1}.
%For the second inequality, if we suppose for concreteness that $a<b$, 
%then  $\{(t,x) : a<t<b, x\in A_{t-a} \}$ has locally finite perimeter, so that the divergence theorem is available and the standard argument is justified.
\end{proof}

\begin{lem}
Let $U$ be a smooth solution of \eqref{ahm1g}-\eqref{ahm2g}
on
$(-T_1,T_1)\times S^1 \times B_\nu(\rho_0)$. Then for the number
$\rho_1\in (0,\rho_0]$ from Proposition \ref{prop3},
given $-T_1\le a < b\le T_1$ and  $\chi\in W^{1,\infty}((-T_1,T_1)\times S^1\times B_\nu(2\rho_1))$, if  $\chi(s,\cdot) $ has compact support in $S^1\times B_\nu(2\rho_1)$
for every $s\in [a,b]$, then
\begin{multline}
\left|
\int_{\{b\}\times S^1\times B_\nu(2\rho_1) } e_\e(U,g) \chi  \ -
\int_{\{a\}\times S^1\times B_\nu(2\rho_1) } e_\e(U,g) \chi  \ 
\right|  \\
 \le C \int_{(a,b) \times S^1\times B_\nu(2\rho_1)}
 e_\e(U,g) \left (|\chi| +   |D\chi| \right).
\label{std_v_est}\end{multline}
\label{standard.ee2}\end{lem}

\begin{proof}
We fix $\rho_1$ as in the proof of Proposition \ref{prop3} so that
\eqref{rho1} holds. Then the proof of \eqref{std_v_est}
is exactly like the proof of  \eqref{ee1} in Lemma \ref{standard.ee1} above, except that
we use for example \eqref{nrg4} and \eqref{rho1} in place of their counterparts in standard
coordinates.
\end{proof}

We will also often use the fact that there exists some $C>0$ such that
\be
\frac 1C \le |\det D\psi(y)| = \sqrt{-g(y)}   \le C 
\label{Jpsi}\ee
for all $y\in (-T_1,T_1)\times S^1\times B_\nu(\rho_0)$.
This is a straightforward consequence of the
definition of $\psi$. A similar estimate holds for the restriction of 
$\psi$ to $\{ y : y^0=0\}$, which we will call $\psi^0$.

\begin{proof}[Proof of Proposition \ref{mainprop2}]

{\bf 1}. Recall that we defined $\zeta_i(s)$ for $i=1,2,3$ in the statement of
Proposition \ref{prop3}. 
Comparing these with the definitions of $\tilde\zeta_i(s)$, we see
that $\zeta_2(s) = \tilde \zeta_2(s)$ and
$\zeta_1(0) \le \tilde\zeta_1(0) + C\tilde\zeta_4(0)$, using \eqref{Jpsi} and
a change of variables. 
Thus $\zeta_i(0)\le C \zeta_0$ for $i=1,2$, and  then
Proposition \ref{prop3} immediately implies that
\be
\tilde \zeta_i(s) \le \zeta_i(s) \le C \zeta_0
\qquad\qquad
\mbox{ for $i=1,2,3$ and $0\le s \le s_1 := \min \{T_1, \rho_1/2c_*\}$}
\label{mt2.00}\ee
In particular, if we write $$B_{\nu}(\rho' \setminus \rho) := B_\nu(\rho')\setminus B_\nu(\rho),$$ then
the definition of $\zeta_3$ implies that
\be
\int_{\{s\} \times S^1 \times B_\nu(\frac {\rho_1} {2} \setminus \frac {\rho_1}{4})}
e_{\e,\lambda}(U, g) \le C(\rho_1)\zeta_3(s) \ \le \ C \zeta_0 \qquad\mbox{ for }s\in [0,s_1].
\label{mt2.0}\ee
Then it follows from a change of variables and \eqref{comparable}, \eqref{Jpsi} that
\be
\int_{ \psi([0,s_1]\times S^1 \times B_\nu(\frac{\rho_1}{2}\setminus\frac{ \rho_1}{4}))}
e_{\e,\lambda}(\calU, \eta) \ \le C \zeta_0.
\label{mt2.1}\ee

{\bf 2}. Next we consider the standard coordinate system, and we  show that the energy of $\calU$ 
is small away from $\Gamma$ for all $t\in [0,t_1]$,
for some $t_1>0$. The idea is to apply
Lemma \ref{standard.ee1} with $a=0$ and $b=t\in (0,t_1]$
and a suitable cutoff function $\chi$, and to use \eqref{mt2.1} to estimate the terms
appearing on the right-hand side of \eqref{ee1}.

To carry this out, let $\chi: [-T_0,T_0]\times \R^3\to \R$ be a smooth function
such that 
\begin{align*}
\chi &= 1\mbox{ on } 
 ([-T_0,T_0]\times \R^3)\setminus \calN_1
\\
\chi &= 0 \mbox{ on }
([-T_0,T_0]\times \R^3) \cap\psi \left( (-T_1,T_1)\times S^1\times B_\nu(\frac{\rho_1}4) \right).
\end{align*}
Thus $D\chi$ is supported on
$([-T_0,T_0]\times \R^3) \cap\psi \left( (-T_1,T_1)\times S^1\times B_\nu(\frac{\rho_1}{2}\setminus\frac{ \rho_1}{4}) \right)$.

Next, the construction of $\psi$ implies that $\partial_0\psi >0$ everywhere
and that $\psi$ maps the set $\{y :y^0=0\}$ into the set $\{(t,x): t=0\}$,
and these facts imply that there exists some $t_1>0$ such that
\be
\mbox{supp}(D\chi) \cap ( [0,t_1]\times \R^3) \subset 
 \psi([0,s_1]\times S^1 \times B_\nu(\frac{\rho_1}{2}\setminus\frac{ \rho_1}{4})).
\label{mt2.2}\ee
Now we apply Lemma \ref{standard.ee1} to find that for $t\in (0,t_1]$,
\[
\int_{\{t\}\times \R^3} e_{\e,\lambda}(\calU, \eta) \chi 
\ \le 
\ \int_{\{0\}\times \R^3} e_{\e,\lambda}(\calU,\eta) \chi 
+\int_{[0,t_1]\times \R^3} e_{\e,\lambda}(\calU,\eta) |D\chi|.
\]
Moreover, \eqref{mt2.2} and \eqref{mt2.1} imply that 
\[
\int_{[0,t_1]\times \R^3} e_{\e,\lambda}(\calU,\eta) |D\chi|
\ \le \ 
\| D\chi\|_\infty
\int_{ \psi([0,s_1]\times S^1 \times B_\nu(\frac{\rho_1}{2}\setminus\frac{ \rho_1}{4}))} e_{\e,\lambda}(\calU,\eta)
\ \le \  
C \zeta_0.
\]
And using properties of the support of $\chi$ with \eqref{mt2.1} and definition of $\zeta_{0}$,
\begin{align*}
\int_{\{0\}\times \R^3} e_{\e,\lambda}(\calU,\eta) \chi  \ 
&= 
\ \tilde\zeta_4(0) + 
\int_{\psi( \{0\} \times S^1 \times B_\nu(\frac{\rho_1}{2}\setminus\frac{ \rho_1}{4}))} e_{\e,\lambda}(\calU,\eta)\le C \zeta_0.
\end{align*}

Since $\chi =1$ on the complement of $\calN_1$, it follows that 
\be
\tilde\zeta_4(t) = \int_{(\{t\}\times \R^3)\setminus \calN_1} e_{\e,\lambda}(\calU, \eta)
\le C \zeta_0 \quad\mbox{ for every $t\in [0,t_1]$}.
\label{mt2.3}\ee

{\bf 3}. Now for $\sigma\ge 0$, define
\begin{align*}
\zeta_1(s;\sigma) &:= \left.\int_{S^1} \left( \int_{B_\nu(\rho_1-c_*s)} (1+\kappa_2|\yn|^2)e_{\e,\lambda}(U) d\yn  - \calE^m_\lambda \right) dy^1\right|_{y^0=s+\sigma} % &\  \le C\zeta_0
\\
\zeta_2(s;\sigma) &:= D_m(U(s+\sigma ,\cdot); \frac 12 \rho_1) %&\ \le C\zeta_0
\\
\zeta_{3} (s;\sigma)&:=
\left.\int_{S^1}\int_{B_\nu(\rho_1-c_*s)} \Big(\abs{D_{\tau} v}^{2}+\e^{2}\abs{F_{\tau}}^{2}+\abs{y^{\nu}}^{2}e_{\e,\lambda}^{\nu}(U) \Big)
 d\yn \ dy^1\ \ \right|_{y^0 = s+\sigma}  .%&\ \le C\zeta_0
\end{align*}
We claim that there exists some $\sigma_1\in (0,s_1]$ and a constant $C>0$ so that 
\be
\mbox{$\zeta_i(0;\sigma_1) \le C\zeta_0$ for $i=1,2,3,$.}
\label{mt2.4}\ee
This will allow us to apply Proposition \ref{prop3} to extend the weighted energy estimates
for $U$ beyond time $y^0= s_1$.

For $i=2$, note that $\zeta_2(s;\sigma) = \zeta_2(s+\sigma)$, so that in particular
$\zeta_2(0;\sigma) = \zeta_2(\sigma)\le C\zeta_0$ for every $\sigma\in (0,s_1]$,
by \eqref{mt2.00}.
For $i=1,3$, it suffices to find some  $\sigma_1\in (0,s_1]$ such that $\sigma_1< \frac {\rho_1}{4c_*}$ 
and
\be
\left. \int_{S^1}\int_{B_\nu(\rho_1\setminus \frac 34 \rho_1)} e_{\e,\lambda}(U,g ) \right|_{y^0=\sigma_1} \  \le \  C \zeta_0
\label{mt2.5}\ee
since then the definitions, \eqref{rho1} and \eqref{mt2.00} imply that 
\[
\zeta_i(0;\sigma_1) \le  \zeta_i(\sigma_1) +
C\left. \int_{S^1}\int_{B_\nu(\rho_1\setminus\frac34 \rho_1)} e_{\e,\lambda}(U,g ) \right|_{y^0=\sigma_1} \  \le \  C \zeta_0
\]
proving \eqref{mt2.4}.

We will deduce \eqref{mt2.5} by using Lemma \ref{standard.ee2} with a suitable cutoff function $\chi$, and using  \eqref{mt2.3} and
a change of variables to estimate the terms
appearing on the right-hand side of \eqref{std_v_est}.

To carry this out, let $\chi:(-T_1,T_1)\times S^1\times B_\nu(\rho_0)\to [0,1]$
be a smooth cutoff function, indepdendent of $y^0\in (-T_1,T_1)$, with support in $ S^1\times B_\nu(2 \rho_1 \setminus \frac {\rho_1}2)$, and such that $\chi = 1$ on 
$S^1\times B_\nu( \rho_1\setminus \frac 34 \rho_1)$.
We also fix $\sigma_1>0$ so small that
\be
\psi( [0,\sigma_1]\times S^1\times B_\nu(\rho_1 \setminus\frac{ \rho_1}2) ) \subset ([0,t_1]\times \R^3)\setminus \calN_1.
\label{sigma1}\ee
Then 
\begin{align*}
\left. \int_{S^1}\int_{B_\nu( \rho_1\setminus \frac 34\rho_1)} e_{\e,\lambda}(U,g ) \right|_{y^0=\sigma_1} \  
&\le \  
\left. \int_{S^1}\int_{B_\nu(2\rho_1)} e_{\e,\lambda}(U,g ) \chi
\right|_{y^0=\sigma_1} \  
\\
&\le
\left. \int_{S^1}\int_{B_\nu(2\rho_1)} e_{\e,\lambda}(U,g ) \chi
\right|_{y^0= 0} \  \\
&\qquad
+
C(\chi) \int_{ [0,\sigma_1]\times S^1\times B_\nu( 2\rho_1\setminus \frac 12 \rho_1)}
e_{\e,\lambda}(U,g)
\end{align*}
The first term on the right-hand side is bounded by $C (\zeta_3(0) + \tilde\zeta_4(0)) \le C \zeta_0$,
and it follows from \eqref{sigma1}, \eqref{mt2.3}, \eqref{Jpsi} and a change of variables that
the second term on the right-hand side is bounded by $C\zeta_0$. Thus we have proved \eqref{mt2.5}, and hence also \eqref{mt2.4}.

{\bf 4}. It follows from \eqref{mt2.4} and Proposition \ref{prop3} that
\be
\mbox{ $\tilde \zeta_i(s+\sigma_1) \le \zeta_i(s;\sigma_1) \le C \zeta_0$
for $i=1,2,3$ and $0<s \le s_2 = \min \{ \rho_1/2c_*, T_1-\sigma_1\}$}
\label{mt2.6}\ee
Note also that $s_2+\sigma_1  > s_1$ unless $s_1 = T_1$.

We now iterate, using Lemma \ref{standard.ee1} to 
estimate $e_{\e,\lambda}(\calU,\eta)$ on $([0,t_2]\times \R^3)\setminus \calN_1$ for some
$t_2>t_1$, with estimates of the right-hand side of \eqref{ee1} provided by \eqref{mt2.6}; 
and then combining the resulting estimate with Lemma \ref{standard.ee2} and
Proposition \ref{prop3} to extend the weighted energy estimates of $e_{\e,\lambda}(U,g)$ beyond
$y^0 = s_2+\sigma_1$.

To complete the proof of the theorem, then, it suffices only to show that after finitely many iterations
of this argument, one can extend the bounds on $\tilde \zeta_i$ to all 
$0 \le s \le T_1 $ (for $i=1,2,3$) and $0\le t \le T_0$ for $i=4$. (The same conclusions
for $-T_1\le s <0$ and $-T_0 \le t <0$ then follow by time reversal symmetry.) A proof of this may
be found 
in \cite[proof of Theorem 22]{Jerrard09} for somewhat different equations, but exactly the same proof is valid here. 
The point is that the proof only involves piecing together estimates in the standard and normal
coordinate systems, and the algorithm for doing so applies equally to any Lorenz-invariant equation.

(In fact the argument in \cite{Jerrard09} relies on a slightly different and more complicated iteration scheme than the  one suggested above, but it remains true that the arguments there
can be used in this setting with essentially no change.)
\end{proof}

We finally prove our main result.

\begin{proof}[Proof of Theorem \ref{mainthm}] 
 %We begin with some quick remarks.  
We will write $\psi_0(y^1, \yn) = \psi(0,y^1, \yn)$. Recall that by assumption, the minimal surface $\Gamma
= \mbox{image(H)}$ satisfies $\partial_0 H(0, y^1) = (1,0,0,0)$ for every $y^1$. As a result,
the normal vectors
$\bar \nu_i$,  satisfy $\bar \nu_i^0 = 0$ for $i=1,2$, see \eqref{barnu}, and hence the range of $\psi_0$ is an open neighborhood of $\Gamma_0$ in  $\{0\}\times \R^3$. 

Also, if we define $d_0(x) := \mbox{dist}(x,\Gamma_0)$, then we note that 
\be
d_0(\psi_0( y^1, \yn)) = |\yn|.
\label{d0yn}\ee
Indeed, again using the fact that $\partial_0 H(0, y^1) = (1,0,0,0)$, we deduce from \eqref{barnu} that the vectors $\{ \partial_1 H(0,y^1), \bar \nu_1, \bar \nu_2\}$ are orthogonal with respect to the {\em Euclidean} inner product, 
and it follows that $|\yn| = |\psi_0(y^1,\yn) - \psi_0(y^1,0)|$ and also the segment $\psi_0(y^1,\yn)-\psi(y^1,0)$ is orthogonal to $T_{\psi(y^1,0)}\Gamma_0$. Since in addition $|\psi_0(y^1,\yn) - \psi_0(y^1,0)|$
is less than the injectivity radius of $\Gamma_0$ by assumption \eqref{T1aa}  on $\rho_0$,
%{Do we want to keep the injectivity radius? If yes, maybe when we introduce $\rho_{1}$ we could say $\rho_{1}\leq \min(\rho_{0}/2, R_{0})$? }, 
these imply \eqref{d0yn}.

{\bf 1}. We first specify initial data for \eqref{ahm1}-\eqref{ahm2} so
that the constant $\zeta_0$ in Proposition \ref{mainprop2} is small.
The fact that the data needs to satisfy the compatibility
condition \eqref{comp} for well-posedness 
means that this is not completely straightforward.

First, let $U^m_{1,\lambda}$ denote a fixed solution of the minimization problem
\eqref{2dmin} for $\e = 1$, and for general $\e>0$,
let $U^m_{\e,\lambda}$ denote the solution obtained by rescaling $U^m_{1,\lambda}$,
so that
$\phi^\e(y) := \phi(\frac y \e), A^\e(y) := \frac 1{\e} A(\frac y\e)$. 
Then %it is straightforward to can check that
\be
e^\nu_{\e,\lambda}(U^m_{\e,\lambda})(y) = \frac 1{\e^2} 
e^\nu_{1,\lambda}(U^m_{1,\lambda})(\frac y \e),
\qquad
\omega(U^m_{\e,\lambda})(y) = \frac 1{\e^2} \omega(U^m_{1,\lambda})(\frac y\e)
\label{scale2}\ee
A useful property of $U^m_{1,\lambda}$ is
\be\label{exp_decay}
e^\nu_{1,\lambda}(U^m_{1,\lambda})(\yn) \le C e^{- c |\yn|}.
\ee
This is proved in  \cite{JaffeTaubes}, Chapter 3, Theorem 8.1 for arbitrary finite-energy critical
points of the $2d$ Euclidean abelian Higgs action.
It follows that
\[
\int_{ \partial B_{\nu}(\rho_1/3) } e^\nu_{\e,\lambda}(U^m_{\e,\lambda}) 
\le Ce^{-c/\e}.
\]
Then by the construction  in the proof of Lemma \ref{ext},
we can modify $U^m_{\e,\lambda}$ on $\R^2\setminus B_\nu(\rho_1/3)$ 
to produce a new configuration $\tilde U^m_{\e,\lambda}$ such that
\be
\tilde U^m_{\e,\lambda} = U^m_{\e,\lambda} \mbox{ in }B_\nu(\rho_1/3),
\qquad\qquad
%e^\nu_{\e,\lambda}(\tilde U^m_{\e,\lambda}) = 0\mbox{ in }\R^2\setminus B(s+\e)
\int_{ B_\nu(\rho_1/2) \setminus B_\nu(\rho_1/3)} e^\nu_{\e,\lambda}(\tilde U^m_{\e,\lambda})
\le
C e^{-c/\e} \le  C \e^2
 \label{tildeU.2}\ee
and in addition $\tilde U^m_{\e,\lambda}$ has the form
\be
\tilde \phi^m_{\e,\lambda} = e^{i \zeta},\quad
\tilde A^m_{\e,\lambda} = d\zeta
\qquad\mbox{ in } \R^2\setminus B_\nu(\rho_1/2)
\label{trivial}\ee
for some smooth function $\zeta$ taking values in $\R/2\pi \Z$, which in particular
implies that
\be\label{tildeU.3}
e^\nu_{\e,\lambda}(\tilde U^m_{\e,\lambda})   = 0 
\qquad\mbox{ in } \R^2\setminus B_\nu(\rho_1/2).
\ee

%Recall the map $\psi_0:S^1\times B_\nu(\rho_0)\to \{0\}\times \R^3$ defined by  $\psi(y^1,\ldots, y^3) = \psi(0, y^1,\ldots , y^3)$, where $\psi$ is defined in \eqref{psi.def}. 
We now define
\be
(\vp^0, \A^0) = (\psi_0^{-1})^*( \tilde \phi^m_{\e,\lambda}, \tilde A^m_{\e,\lambda})
%\vp^0 = \tilde \phi^m_{\e,\lambda} \circ \psi_0^{-1},\qquad \A^0 := (\psi_0^{-1})^* \tilde A^m_{\e,\lambda}
\qquad\qquad\mbox{ on }\psi_0(S^1\times B_\nu(\rho_1)) \subset \{0\}\times \R^3,
\label{phi01}\ee
(where here we view $ \tilde \phi^m_{\e,\lambda}, \tilde A^m_{\e,\lambda}$
as functions on $S^1\times \R^2$ that are independent of $y^1$.)
%Note that $\A^0$ has the form $\A^0 = \sum_{i=1}^3 \A^0_i dx^i$.
%and on $\R^3\setminus \psi_0(S^1\times B_\nu(\rho_1))$ {$\leftarrow$ This is confusing}. 
In view of \eqref{trivial}, in 
$\psi_0( S^1\times ( B_\nu(\rho_1) \setminus B_\nu(\rho_1/2))$,
$\calU^0$ has the form
\[
\vp^0 = e^{i \hat \zeta}, \qquad \A^0 = d\hat\zeta 
\qquad\qquad\mbox{ for }\  \hat \zeta := \zeta \circ \psi_0^{-1}.
\]
Now fix a smooth $q:\R^3\setminus [ \psi_0(S^1\times B_\nu(\rho_1/2))]\to \R/2\pi \Z$  
such that $q =\hat \zeta$ on the intersection of the domains of $q$ and  $\hat \zeta$,
and define
 \be
\vp^0 = e^{i q}, \qquad \A^0 = dq \qquad\mbox{ in }\R^3\setminus \psi_0(S^1\times B_\nu(\rho_1/2)).  
\label{phi0.2}\ee
The choice of $q$ implies that \eqref{phi01} and  \eqref{phi0.2} are consistent.

Finally, we let $\calU$ denote the solution of \eqref{ahm1}-\eqref{ahm2}
in the temporal gauge $\A_0 = 0$, 
with initial data
\be\label{phi.t}
(\vp, \A_1,\ldots, \A_3)|_{t=0} = 
(\vp^0, \A_1^0,\ldots, \A_3^0) =: \calU^0,\quad\qquad
\partial_{t} (\vp, \A_1,\ldots, \A_3)|_{t=0} = 0.%,\quad\mbox{and}\quad \A_{0}=0.
\ee
The existence of the solution $\calU$ follows from the discussion in Sections \ref{s:gwp}-\ref{s:id}.

 {\bf 2}. We claim that in this gauge, and for the initial data  \eqref{phi01}, \eqref{phi0.2} and \eqref{phi.t} above,
\be
\tilde \zeta_i(0)\le C \e^2 \qquad \mbox{ for } \ \ i=1,2,4
\label{zetaoorderep}\ee 
We remind the reader that these quantities are defined in the statement of Proposition \ref{mainprop2}.

First, it is routine to check from \eqref{phi0.2}, \eqref{phi.t} that $e_{\e, \lambda}(\calU,\eta) = 0$ in $\R^3\setminus \psi_0(S^1\times B_\nu(\rho_1/2))$, and hence that $\tilde \zeta_4(0) = 0$.

The quantities $\tilde \zeta_1, \tilde \zeta_2$ are expressed in terms of the
$U =(\phi, A) = \psi^*\,\calU$. %that is, the solution written in terms of the normal coordinate system,and $\tilde \zeta_1, \tilde \zeta_2$  involve the behaviour of $U$ only on $\{0\}\times S^1\times B_\nu(\rho_1/2)$.
So we must translate our assumptions about $\calU$ at $t=0$ 
into information about $U$ for $y^0=0$.

Let us write $A^0 ( y^1,\yn) := \sum_{i=1}^3 A_i(0, y^1,\yn) dy^i$ to denote the spatial part of $A$ at time $y^0=0$, and $U^0 = (\phi^0,A^0),$ where $\phi^0( y^1,\yn) := \phi(0, y^1,\yn)$. This says that
$U^0 = i^* U$, for $i(y^1,\yn) = (0,y^1, \yn)$. As a result,
\[
U^0 = i^*\psi^* \, \calU = (\psi\circ i)^* \calU = \psi_0^*\,\calU = \psi_0^*\, \calU^0,
\] 
using the fact for the final equality that $\mbox{Image}(\psi_0)\subset \{0\}\times \R^3$.
Hence
\eqref{phi01} implies that $U^0 = ( \tilde \phi^m_{\e,\lambda}, \tilde A^m_{\e,\lambda})$, 
or in other words that
\begin{align*}
\phi(0, y^1,\yn) 
&= \tilde \phi^m_{\e,\lambda}(\yn), \\
\sum_{i=1}^3 A_i(0, y^1, \yn) dy^i 
%= \tilde A^{m,1}_{\e,\lambda}(\yn) dy^{\nu 1} + \tilde A^{m,2}_{\e,\lambda}(\yn)dy^{\nu 2}.
&= \tilde A^{m,1}_{\e,\lambda}(\yn) dy^{2} + \tilde A^{m,2}_{\e,\lambda}(\yn)dy^{3}.
\end{align*}
As a result,
\be\label{spatial}
D_1\phi = 0, \qquad
F_{1j} = F_{j1} = 0\mbox{ for }j=2,3
\ee
when $y^0=0$, everywhere in $S^1\times B_\nu(\rho_1)$.
Next, we can write the identity $U = \psi^* \calU$ explicitly as
\be
\phi = \varphi\circ \psi, \qquad
A_\alpha = \frac{\partial \psi^\mu }{\partial y^\alpha} \A_\mu\circ \psi
\label{explicit.pullback}\ee
From these we check that
\[
D_{y^0} \phi  
\ = \   
\frac{\partial \psi^\mu }{\partial y^0} \  (D_{x^\mu} \varphi) \circ\psi
\ = \   
\sum_{k=1}^3 \frac{\partial \psi^k}{\partial y^0}
 (D_{x^k} \varphi) \circ\psi
\]
due to the temporal gauge and the initial condition $\partial_{x^0}\varphi = 0$.
(Here for example
$D_{y^0} = \frac{\partial}{\partial y^0} - i A_0$.)
Recall that we have assumed that the initial velocity of $\Gamma$ vanishes.
This states that $\partial_{y^0}h(0, y^1) = 0$, and then it follows from the
explicit form \eqref{psi.def} of $\psi$ that $\frac{\partial \psi^k}{\partial y^0}(0, y^1, \yn) 
= O(|\yn|)$ for $k=1,2,3$.
Thus for $y^0=0$,
\be
|D_{y^0}\phi|^2 
\le C |\yn|^2 \sum_{k=1}^3 |(D_{x^k} \varphi)\circ \psi|^2 
\le C |\yn|^2 \sum_{k=1}^3 |D_{y^k}\phi|^2
\overset{\eqref{spatial}}
\le 
C |\yn|^2  \ |D_\nu \phi|^2.
\label{tmprl1}
\ee
Similarly, \eqref{explicit.pullback}, the temporal gauge, and the initial conditions  imply that
for $y^0=0$,
\[
F_{\a\b} = \sum_{i,j=1}^3 \frac{\partial \psi^i}{\partial y^\alpha}\frac{\partial \psi^j}{\partial y^\b}
( \frac {\partial \A_i}{\partial x^j} - \frac{\partial \A_j}{\partial x_i})\circ \psi,
\]
so, again using the fact that $\frac{\partial \psi^k}{\partial y^0}(0, y^1, \yn) 
= O(|\yn|)$, we see that
\be
F_{0k}^2 
\le 
C|\yn|^2 \sum_{i,j=1}^3 \F_{ij}^2\circ \psi 
\le 
C|\yn|^2 \sum_{i,j=1}^3 F_{ij}^2 
\overset{\eqref{spatial}}
\le C |\yn|^2 |F_\nu|^2,
\qquad k=1,2,3.
\label{tmprl2}
\ee
Combining \eqref{spatial}, \eqref{tmprl1} and \eqref{tmprl2}, and recalling
\eqref{pos2}, we find that
\[
e_{\e,\lambda}(U, g)(0, y^1, \yn)  \leq (1+ C\abs{\yn}^2) e^\nu_{\e,\lambda}(\tilde U^m_{\e,\lambda})(\yn)
\]
for all $y^1\in S^1$.
We have chosen $U^m_{\e,\lambda}$ exactly so that
it satisfies  $\int_{\R^2}e^\nu_{\e,\lambda}(U^m_{\e,\lambda}) = \calE^m_\lambda$, 
so we can use  \eqref{tildeU.2}, \eqref{tildeU.3}, together with \eqref{scale2} and a change of variables, to find  that
\begin{align*}
\tilde \zeta_1(0)&\le C\e^2+C\int_{S^1 \times B_\nu(\rho_1/2)}  \abs{\yn}^2 
e^\nu_{\e,\lambda}(\tilde U^m_{\e,\lambda})(\yn)
dy^1d\yn \\
&\le C \e^2 + C \e^2 \int_{\R^2}|\yn|^2 e^\nu_{1,\lambda}(U^m_{1,\lambda})  d\yn\\
&
\le C \e^2.
\end{align*}
The finiteness of the second moment in the last inequality follows from 
the  exponential decay estimate  \eqref{exp_decay}.

Next, since $|\omega(U)|\le C(\lambda)e^\nu_{\e,\lambda}(U)$, and since
$\int_{\R^2} \omega(U^m_{\e,\lambda}) = \pi m$,  one can check, again
using \eqref{scale2}, \eqref{tildeU.2}, \eqref{tildeU.3}, 
(recall also  the definitions \eqref{D.def}, \eqref{defect}, \eqref{f.def})
that
\[
\tilde 
\zeta_2(0) \le C \e^3 \le C \e^2,
\]
where the scaling $\e^3$ comes ultimately from \eqref{f.def}.

{\bf 3}.  Proposition \ref{mainprop2} now implies
\be\label{4t1}
\tilde \zeta_i(s) \le C \e^2
\mbox{ for $=1,2,3$ and $-T_1\le s\le T_1$},
\ee
and
\be\label{4t1.c2}
\tilde \zeta_4(t) \le C \e^2 \mbox{ for }-T_0\le t \le T_0.
\ee
In particular, \eqref{4t1.c2} implies \eqref{T1.c2}. 
Next, by \eqref{comparable}, \eqref{Jpsi}, and the definition of $d^\nu$,
\begin{align*}
\int_{\mathcal N_1} (d^{\nu})^{2}e_{\e,\lambda}(\mathcal U,\eta) dx dt &\leq C \int_{(-T_{1},T_{1})\times S^{1}\times B_{\nu}(\rho_{1}/2)}\abs{y^{\nu}}^{2}e_{\e,\lambda}(U) dy\\
&\leq 
C\int_{-T_{1}}^{T_{1}}\tilde \zeta_{3}(s) \,ds \\% ds +C\int_{-T_{1}}^{T_{1}}\tilde \zeta_{2}(s) ds\\
&\overset{\eqref{4t1}}{\leq} 
C\e^2.
\end{align*}
proving \eqref{T1.c3}.
(In fact the estimate of $\tilde \zeta_3$ in \eqref{4t1} is substantially stronger than 
\eqref{T1.c3}.)

{\bf 4}.
To establish \eqref{T1.c1}, we carry out a gauge transform to arrange that
\be\label{normal_temp}
A_{0}=0
\quad \mbox{ in }\ \ \psi^{-1}(\calN_1) \subset (-T_1,T_1)\times S^1\times B_\nu(\rho_1/2) .
\ee
Toward this end,  let $\chi$ be a smooth function with support in Image$(\psi)$
such that $\chi=1$ in $\calN_1$, and define
\[
v = (1- \chi) e_0 + \chi \hat v,\qquad e_0 := (1,0,0,0), \qquad
\hat v^\mu := \frac {\partial\psi^\mu}{\partial y^0}\circ \psi^{-1}.
\]
Then from \eqref{explicit.pullback} we see that \eqref{normal_temp} holds
if and only if $v^\mu \A_\mu = 0$ in $\calN_1$. To arrange this,
let $f$ satisfy the linear transport equation
\[
v^\mu( \partial_\mu f +\A_\mu) = 0 \ 
%v^0 \partial_t f +v^i (\partial_{x^i} f + \A_i)= 0
\mbox{ in }\ (-T_0,T_0)\times \R^3, \quad\qquad f(0,x) = 0 \mbox{ for }x\in \R^3.
\]
(The condition $g_{00}<0$ implies that $\hat v$ and hence $v$ are timelike; thus  $v^0$ never vanishes, so the above initial value problem is solvable.) Then after the gauge transform
$
( \vp,\A) \to  (e^{if}\vp, \A+df)$,
the equation that defines $f$ states exactly that the new connection $1$-form satisfies
$v^\mu  \A_\mu=0$. Thus we have achieved \eqref{normal_temp}.
Also, since $\tilde \zeta_{i}, i=1,\dots 4$ are gauge invariant, \eqref{4t1}-\eqref{4t1.c2} still hold.

Recall the form of $U^{\scriptsize N\!O} = (\phi^{\scriptsize N\!O}, A^{\scriptsize N\!O})$:
\[
\phi^{\mbox{\scriptsize N\!O}}(\yt, \yn) = 
\phi^m(\yn),
\qquad
A^{\mbox{\scriptsize N\!O}}(\yt, \yn)
:=  A^{m}_1(\yn) dy^{\nu1} +  A^{m}_2(\yn) dy^{\nu2},
\]
where $U^m = (\phi^m, A^m) = U^m_{\e,\lambda}$  is a minimizer.
We stipulate that $U^m_{\e,\lambda}$ is {exactly} the same minimizer
out of which $(\tilde\phi^m_{\e,\lambda},\tilde A^m_{\e,\lambda})$ is constructed
in Step 1. 
Then by a change of variables and a Poincar\'e inequality  we have 
\begin{align}
\int_{\calN_{1}} 
| \vp-  \vp^{\scriptsize N\!O}|^2& +
\e^2 \sum_{\a=0}^3 | \A_{\a} - \A_{\a}^{\scriptsize N\!O} |^2
\nonumber\\
&\le C 
\int_{(-T_1,T_{1})\times S^{1}\times B_{\nu}(\rho_{1}/2)}
| \phi-  \phi^{\scriptsize N\!O}|^2 +
\e^2 \sum_{\a=0}^3 | A_{\a} - A_{\a}^{\scriptsize N\!O} |^2
\nonumber\\
&\le C 
\int_{(-T_1,T_{1})\times S^{1}\times B_{\nu}(\rho_{1}/2)}
|\partial_{y^0}( \phi-  \phi^{\scriptsize N\!O}) |^2 +
\e^2 \sum_{\a=0}^3 | \partial_{y^0}(A_{\a} - A_{\a}^{\scriptsize N\!O}) |^2
\label{L2est.4}\\
&
\qquad\qquad\qquad\qquad
\ + \ 
\left.C\left(\int_{S^{1}\times B_{\nu}(\rho_{1}/2)}
|\phi-  \phi^{\scriptsize N\!O} |^2 +
\e^2 \sum_{\a=0}^3 |A_{\a} - A_{\a}^{\scriptsize N\!O} |^2\right)\right|^{y^0=T_1}_{y^0=-T_1}.\nonumber
\end{align}
In the first integral on the right-hand side we use the
explicit form of $U^{\scriptsize N\!O}$ and the gauge
$A_0 = 0$ to write
\[
|\partial_{y^0}( \phi-  \phi^{\scriptsize N\!O}) |^2 +
\e^2 \sum_{\a=0}^3 | \partial_{y^0}(A_{\a} - A_{\a}^{\scriptsize N\!O}) |^2
= |D_0 \phi|^2 + \e^2 |F_\tau|^2.
\]
Inserting this into \eqref{L2est.4} and using \eqref{4t1}, we conclude the first integral is bounded by $C\e^2$.
To bound the second integral, using fundamental theorem of calculus we observe 
\begin{align*}
&\left.\left(\int_{S^{1}\times B_{\nu}(\rho_{1}/2)}
|\phi-  \phi^{\scriptsize N\!O} |^2 +
\e^2 \sum_{\a=0}^3 |A_{\a} - A_{\a}^{\scriptsize N\!O} |^2\right)\right|_{y^0=T_1}\\
&\quad \leq C\left.\left(\int_{S^{1}\times B_{\nu}(\rho_{1}/2)}
|\phi-  \phi^{\scriptsize N\!O} |^2 +
\e^2 \sum_{\a=0}^3 |A_{\a} - A_{\a}^{\scriptsize N\!O} |^2\right)\right|_{y^0=0}\\
&\qquad +  C\int_0^{T_1} \int_{S^{1}\times B_{\nu}(\rho_{1}/2)}\abs{\partial_{y^0}\phi}^2 +\e^2\sum_{\a=0}^3(\partial_{y^0}A_\alpha)^2.
\end{align*}
The second integral is bounded again using $A_0=0$ and \eqref{4t1}, whereas $C\e^2$ bounds for the first integral follow from the construction of the data. 
The boundary term $y^0=-T_1$ is treated exactly the same.  This gives  \eqref{T1.c1}.
%It follows from Step 2 and Proposition \ref{mainprop2}
%that \eqref{mt1.c2}, \eqref{mt1.c3} holds with $\zeta_0 = C \e$.
%It remains to verify that for the choice of initial data considered 
%here,
%$\calU^{\mbox{\scriptsize N\!O}}$ admits the description
%given in \eqref{N-O} in the set $\calN_1$.
%In \eqref{mt1.c3},
%we take $Y^\nu$ to be the final two components of $\psi^{-1}$,
%so that $Y^\nu(\psi(y)) = \yn$ for $y\in (-T_1,T_1)\times S^1\times B_\nu(\rho_0)$.
%It then suffices to verify that in $\calN_1$, the equations
%that characterize $\calU^{\mbox{\scriptsize N\!O}}$
%are satisfied by  $(Y^\nu)^*U^m$.
%We recall that these equations 
%consist of the gauge condition $v^\a\A_\a = 0$ and initial conditions
%$(\vp^{\mbox{\scriptsize N\!O}}, \A^{\mbox{\scriptsize N\!O}})(0,x) = (\vp, \A)(0,x)$ 
%together with
%\[
%v^\alpha D_\alpha \vp^{\mbox{\scriptsize N\!O}}  = 0, \qquad
%v^\a \F_{\a i}^{\mbox{\scriptsize N\!O}} = 0 \  \mbox{ for }i=1,2,3.
%\]
\end{proof}

\bibliography{ahm}
\bibliographystyle{plain}
 
%\vspace{.125in}

\end{document}